\documentclass[11pt]{article}
\usepackage[margin=1in]{geometry} 
\usepackage{graphicx} 

\usepackage{amsmath,amsthm,amssymb,amsbsy}  
\usepackage{comment}
\usepackage{hyperref}
\usepackage{url}

\hypersetup{
	colorlinks=true, 
	linkcolor=blue,
	citecolor=blue,
	filecolor=blue,
	urlcolor=blue,
}

\usepackage{algorithm,algorithmic}

\numberwithin{equation}{section}

\theoremstyle{plain}
\newtheorem{theorem}{Theorem}[section]
\newtheorem{proposition}[theorem]{Proposition}

\theoremstyle{definition}
\newtheorem{assumption}{Assumption}

\theoremstyle{remark}
\newtheorem{remark}[theorem]{Remark}

\usepackage{bm, bbm}   

\newcommand{\Q}{\mathbb{Q}}
\newcommand{\R}{\mathbb{R}}

\newcommand{\bE}{\mathbb{E}}
\newcommand{\dd}{\mathrm{d}}

\newcommand{\Om}{\Omega}

\newcommand{\Ac}{\mathcal{A}}
\newcommand{\Fc}{\mathcal{F}}
\newcommand{\Gc}{\mathcal{G}}
\newcommand{\Lc}{\mathcal{L}}
\newcommand{\Pc}{\mathcal{P}}

\newcommand{\Xc}{\mathcal{X}}
\newcommand{\F}{\mathbb{F}}
\newcommand{\G}{\mathbb{G}}
\newcommand{\bP}{\mathbb{P}}
\newcommand{\x}{\times}

\DeclareMathOperator{\Var}{Var}

\usepackage[dvipsnames]{xcolor}

\title{Particle Methods with Deep Learning for Stochastic Control \\ under Partial Observation}
\author{Mathieu Lauri\`{e}re \thanks{Shanghai Center for Data Science; NYU-ECNU Institute of Mathematical Science, NYU Shanghai. mathieu.lauriere@nyu.edu}
	\and Xiaolu Tan \thanks{Department of Mathematics, The Chinese University of Hong Kong. xiaolu.tan@cuhk.edu.hk}
	\and Jiefei Yang \thanks{NYU-ECNU Institute of Mathematical Sciences, NYU Shanghai. jy5595@nyu.edu.}
}
\date{}

\begin{document}
	
	\maketitle
	
	\begin{abstract}
				Numerical computation of stochastic control problems under partial observation is challenging because the dynamic programming formulation is naturally posed on the conditional distribution of the hidden state. We propose particle-based methods that replace this infinite-dimensional filtering state by a finite-dimensional weighted particle system, building on recent limit theory for mean-field control with common-noise-adapted controls. We prove, under suitable assumptions, convergence of the fully discretized particle approximation to the original continuous-time partially observed control problem. The particle reformulation is high-dimensional but permutation-invariant, a structure that can be exploited by symmetric neural network architectures. We develop two deep learning algorithms: a direct optimization method for feedback controls and a Deep BSDE method for particle problems admitting a backward stochastic differential equation representation. We also extend the computational framework to partially observed mean-field control problems, which have been studied theoretically but remain less developed numerically. Numerical experiments on a linear--quadratic benchmark, a nonlinear partially observed mean-field control problem, and two financial applications, portfolio liquidation and asset allocation, demonstrate the accuracy and practical utility of the approach. 
	\end{abstract}

    {\bf Key words.} Stochastic optimal control; Partial observation; Deep learning; Particle method
    
    {\bf MSC codes.} 65C35, 49M25, 93C41
	
	\section{Introduction}
	Stochastic optimal control provides a mathematical framework for sequential decision-making under uncertainty, with applications ranging from engineering and operations research to finance and economics. In many such problems, the controller does not observe the full state of the system. Decisions must instead be made from a noisy signal, while relevant state variables remain hidden: examples include control from imperfect sensor measurements, portfolio optimization under an unobserved drift, and trading problems in which latent market factors must be inferred from prices. This practical feature has motivated a large literature on stochastic control under partial observation, including early work on dynamic programming and measure-valued formulations~\cite{davis1973dynamic, fleming1980measure}, as well as the monograph~\cite{bensoussan1992stochastic}; see also~\cite{bensoussan1985optimal, pham2008portfolio, balata2019class} for representative theoretical and financial developments.
	
	Partial observation also changes the computational nature of the control problem. In the fully observed setting, the dynamic programming principle (DPP) typically leads to a Hamilton--Jacobi--Bellman (HJB) equation whose state variable is the current physical state. Under partial observation, however, the relevant state is the conditional distribution of the hidden state given the observations. This filtering distribution is measure-valued, so the separated control problem is fully observed but infinite-dimensional. Consequently, classical DPP-based approaches lead to HJB equations on spaces of probability measures or to infinite-dimensional filtering equations~\cite{fleming1980measure, lions1989viscosity, gozzi2000hamilton}. Although this viewpoint is fundamental, it poses a serious obstacle to direct numerical implementation. The relaxed formulation introduced in~\cite{el1988existence} provides another classical perspective, and a recent discrete-time approximation was studied in~\cite{li2024discrete}.
	
	The aim of this paper is to develop particle-based numerical methods for continuous-time stochastic control under partial observation, with convergence guarantees for the standard control problem and implementable deep learning algorithms. Our starting point is the classical reference probability reformulation: after a change of measure, the observation process becomes a Brownian common noise and the likelihood process carries the information from the observations. This places the problem in a form closely related to mean-field control (MFC) with controls adapted to common noise. We exploit this connection by using the recent limit theory of~\cite{bouchard2025limit} to replace the infinite-dimensional filtering state by a finite particle system with likelihood weights.
	
	The resulting approximation has two features that are particularly useful for computation. First, after time discretization, the partially observed problem is transformed into a finite-dimensional, fully observed Markov control problem. Second, although the dimension grows with the number of particles, the system is invariant under permutations of the particles. This symmetry is essential for computation: it allows us to use symmetric neural network architectures, such as DeepSets-type architectures~\cite{germain2022deepsets}, to represent feedback controls and BSDE integrands with a parameter count that does not scale as a generic high-dimensional network would. In this way, the particle approximation supplies both a mathematically justified finite-dimensional surrogate and a structure that can be exploited algorithmically.
	
	Our contributions can be summarized as follows.
	\begin{itemize}
		\item {\bf Convergence theory.} For standard stochastic optimal control under partial observation, we establish convergence from the original continuous-time problem to a fully discretized, fully observed $N$-particle control problem. The result combines the particle limit from~\cite{bouchard2025limit} with a time-discretization argument for the controlled particle system.
		
		\item {\bf Particle approximation for partially observed MFC.} We show how the same particle representation extends naturally to MFC problems under partial observation, a class of problems studied theoretically in~\cite{buckdahn2017mean, wan2025discrete, bayer2024continuous, fuhrman2026optimal} but still rarely addressed numerically. We present the algorithm for the MFC extension and numerical examples.
		
		\item {\bf Deep particle algorithms.} We develop two deep learning algorithms for the particle control problem: a direct method that optimizes a symmetric feedback policy, and a Deep BSDE method for the case where the particle problem admits a BSDE characterization. We also implement a recurrent neural network (RNN) method adapted from~\cite{han2021recurrent} as a baseline for path-dependent controls. Numerical experiments show that the particle approach is accurate on a linear--quadratic benchmark with an explicit solution, remains consistent across the direct and Deep BSDE solvers for a nonlinear partially observed MFC example, and is effective on portfolio liquidation and asset allocation problems.
	\end{itemize}
	
	\paragraph{Related works.}
	\begin{sloppypar}
	Our work is connected to three strands of literature. The first is the classical theory of stochastic control under partial observation and its numerical approximation. As discussed above, this includes dynamic programming and filtering formulations~\cite{davis1973dynamic, fleming1980measure, bensoussan1992stochastic}, infinite-dimensional HJB equations~\cite{lions1989viscosity, gozzi2000hamilton}, relaxed formulations~\cite{el1988existence}, financial applications~\cite{pham2008portfolio, balata2019class}, and recent discrete-time approximation schemes~\cite{li2024discrete}.
	
	The second strand is the recent literature on partially observed MFC and conditional McKean--Vlasov dynamics. Partially observed MFC was introduced in~\cite{buckdahn2017mean}, which established wellposedness and a Pontryagin maximum principle. Subsequent developments include discrete-time partially observed MFC~\cite{wan2025discrete}, continuous-time control with discrete-time observations~\cite{bayer2024continuous}, and hidden Markov switching models~\cite{fuhrman2026optimal}. The associated conditional McKean--Vlasov stochastic differential equations and their particle approximations have been studied in~\cite{buckdahn2023general, du2024particle}. Our particle approximation for the standard control problem relies instead on the common-noise-adapted MFC limit theory of~\cite{bouchard2025limit}.
	\end{sloppypar}
	
	\begin{sloppypar}
	The third strand concerns deep learning methods for high-dimensional stochastic control and related equations. Direct neural-network approaches for stochastic control were developed in~\cite{han2016deep}. The Deep BSDE method was introduced in~\cite{han2018solving}, with convergence analysis studied in~\cite{han2020convergence}. In mean-field settings, related methods include direct learning~\cite{carmona2022convergence}, population-dependent architectures~\cite{dayanikli2023deep}, and symmetric networks for permutation-invariant problems~\cite{germain2022deepsets}. General deep learning frameworks for stochastic control are studied in~\cite{bachouch2022deep}. Our contribution is to combine these computational tools with a particle reformulation tailored to partial observation.
	\end{sloppypar}
	
	\paragraph{Organization of the paper.}
	Section~\ref{sec:problem-approximation} formulates the partially observed control problem, presents the particle approximation and time discretization, and states the main convergence theorem. It also describes the extension to partially observed MFC. Section~\ref{sec:numerical-schemes} develops the direct and Deep BSDE particle algorithms. Section~\ref{sec:numerical-experiments} evaluates these schemes on a linear--quadratic benchmark, a partially observed MFC model, and two financial portfolio optimization problems. Section~\ref{sec:ccl} concludes the paper.

	\section{Problem formulation and particle approximation} \label{sec:problem-approximation}
	
	Let us introduce a general optimal control problem with partial observations in a continuous-time setting with finite horizon $T>0$, c.f. \cite{bensoussan1992stochastic},
	and then provide some approximation algorithms based on the particle method.
	We then also discuss its extension to the MFC problem with partial observation in the spirit of \cite{buckdahn2017mean}.
	Throughout the paper, we set $\mathcal{X}=\R^d$ as state space,  $A \subset \R^d$ the action space, and $\mathcal{U} = \R^d$ the observation space.
	
	\subsection{An optimal control problem with partial observation} \label{subsec:problem}
	
	Let $d \ge 1$, we are given the coefficient functions $(b, \sigma, \sigma_0, h, f): [0, T] \times \R^d \times A \longrightarrow \mathbb{R}^d\times \mathbb{R}^{d\times d}\times \mathbb{R}^{d\times d} \times \R^d \times \R$, $g: \R^d \longrightarrow \R$,
	and an initial distribution $\mu_0$ on $\R^d$.
	Given a control process $\alpha$, the controlled (unobservable) state process $X$ is given by
	\begin{equation}
		X_t 
		= 
		X_0
		+
		\int_0^t  b(s, X_s, \alpha_s)\,\dd s 
		+
		\int_0^t \sigma(s, X_s,  \alpha_s) \dd W_s 
		+
		\int_0^t \sigma_0(s, X_s, \alpha_s)\dd B_s,
		~t \in [0,T],
	\end{equation}
	where $\mathcal{L}(X_0) = \mu_0$, $W$ and $B$ are two independent $\R^d$-valued standard Brownian motions in a filtered probability space $(\Omega, \mathcal{F}, \F = (\Fc_t)_{0 \le t \le T},  \mathbb{P})$.
	While the controlled state process $X$ is unobservable, one can observe the process $U$ given by
	\begin{equation}
		U_t 
		=
		\int_0^t h(s, X_s, \alpha_s)\,\dd s 
		+ 
		B_t,
		~
		t \in [0,T]. 
	\end{equation}
	The control process $\alpha$ is required to be adapted to the observation filtration $\G$ generated by the observable process $U$.
	We then obtain the following optimal control problem with partial observation:
	\begin{equation} \label{eq:original-prob}
		V_0 
		~:=~
		\inf_{\alpha ~\mbox{is predictable to}~\G} 
		~\mathbb{E} \left[ \int_0^T f(t, X_t,  \alpha_t)\,\dd t + g(X_T) \right].
	\end{equation}
	Our main objective is to provide some approximation algorithms.
	
	\paragraph{Reformulation by reference probability approach.}
	
	Following the classical reference probability approach, let us introduce an equivalent reformulation of problem by considering the probability measure $\mathbb{Q}$ defined by 
	\begin{equation} \label{eq:change-of-measure}
		\frac{\dd \mathbb{Q}}{\dd \mathbb{P}}
		:= \big(L_T \big)^{-1}, 
		~\mbox{with}~
		L_t:= \exp\left( \int_0^t h(s,X_s, \alpha_s)\,\dd U_s - \frac{1}{2}\int_0^t \big| h(s, X_s, \alpha_s) \big|^2\,\dd s\right),
		~t \in [0,T].
	\end{equation}
	Under the probability measure $\mathbb{Q}$, the process $U$ becomes a standard Brownian motion, independent of the Brownian motion $W$, the control process $\alpha$ is adapted to the filtration generated by the Brownian motion $U$.
	
	\vspace{0.5em}
	
	In the probability space $(\Om, \Fc, \Q)$, the processes $W$ and $U$ are two independent standard Brownian motions,
	the (unobservable) state process $X$ satisfies $\mathcal{L}^{\Q}(X_0) = \mu_0$ and
	\begin{align} \label{eq:SDE_X}
		X_t 
		=~
		X_0 
		& +
		\int_0^t \big(b(s, X_s,  \alpha_s) - \sigma_0(s, X_s, \alpha_s ) h(s,X_s, \alpha_s) \big) \dd s \nonumber \\
		& +
		\int_0^t \sigma(s, X_s,  \alpha_s)\dd W_s 
		+ 
		\int_0^t \sigma_0(s, X_s, \alpha_s)\dd U_s, 
		~~t \in [0,T],
	\end{align}
	and $L$ is defined in \eqref{eq:change-of-measure}.
	Denoting by $\mathcal{A}$ the space of all control processes $\alpha$ which are predictable to the filtration $\mathcal{G}$ generated by the Brownian motion $U$,
	we obtain the equivalent formulation of the initial optimal control problem with partial observation:
	\begin{equation} \label{eq:reform_V0}
		V_0 
		~=
		\inf_{\alpha \in \mathcal{A}} ~
		\mathbb{E}^{\mathbb{Q}} 
		\left[ 
		\int_0^T L_t ~f(t, X_t,  \alpha_t)\,\dd t 
		+
		L_T~ g(X_T) 
		\right]. 
	\end{equation}
	
	Let us assume the following technical conditions on the coefficient functions.
	
	\begin{assumption} \label{assum:coeff}
		The action space $A \subset \R^d$ is compact.
		Further, there exist constants $p > q \ge 1$ such that $\mu_0$ has finite $p$-moment, i.e. 
		$$
		\int_{\R^d} |x|^p \mu_0(dx) < \infty.
		$$
		The coefficient functions $(b, \sigma, \sigma_0, h)$ are all bounded continuous, and there exists a constant $C> 0$ such that
		$$
		\big| (b, \sigma, \sigma_0, h) (t,x,a) - (b, \sigma, \sigma_0, h)(t,y,a) \big| 
		\le 
		C |x-y|,
		$$
		for all $(t,x,y,a) \in [0,T] \x \R^d \x \R^d \x A$.
		Moreover, the coefficient functions $f$ and $g$ are continuous, and there exists a constant $C> 0$ such that
		\begin{equation} \label{eq:growth_rg}
			\big| f(t,x,a) \big| + \big| g(x) \big| 
			\le
			C \big(1+ |x|^q \big),
			~~\mbox{for all}~
			(t,x,a) \in [0,T] \x \R^d \x A.
		\end{equation}
	\end{assumption}
	
	\begin{remark}
		Under Assumption \ref{assum:coeff}, it is clear that the SDE \eqref{eq:change-of-measure}-\eqref{eq:SDE_X} is wellposed.
		Moreover, it follows by standard analysis that 
		$$
		\sup_{\alpha \in \mathcal{A}}
		\mathbb{E}^{\bP} \bigg[ \sup_{0 \le t \le T} \big| X_t \big|^p \bigg] < \infty,
		$$
		so that the value function $V_0$ is finite under the growth condition \eqref{eq:growth_rg}.
	\end{remark}

	\begin{remark}
		We assume the compactness of $A$ and the boundedness of  coefficients $(b, \sigma, \sigma_0)$  for simplicity.
		One can also assume polynomial growth conditions on the coefficient functions, together with appropriate integrability conditions on the admissible control processes $\alpha$ (without compactness of $A$),
		see also \cite[Assumption 3.1]{bouchard2025limit} for similar technical conditions in the setting of standard McKean-Vlasov optimal control problems.
	\end{remark}

	\subsection{Approximation by controlled particle system} \label{subsec:particle-approx}
	
	As discussed in \cite[Section 5]{bouchard2025limit}, the reformulation \eqref{eq:reform_V0} of the partially observed control problem can be considered as a special mean-field control problem with common noise adapted control processes,
	so that one can approximate it by a fully observable $N$-particle system.
	We will explore this particle approximation method.

	\paragraph{A continuous-time controlled particle system.}
	
	Let $N \ge 1$, $(\Omega^N, \mathcal{F}^N, \mathbb{Q}^N)$ be a probability space equipped with a Brownian motion $U$, a sequence of i.i.d. random variables $(\xi_k)_{k=1}^N$ with $\xi_k \sim \mu_0$, and a sequence of Brownian motions $(W^k)_{k=1}^N$, all of which are mutually independent. 
	Let us define the filtration $\mathbb{F}^N:=(\mathcal{F}_t^N)_{t\in [0,T]}$ by $\mathcal{F}_t^N:=\sigma(\xi_1,\dots, \xi_N, W_s^1,\dots, W_s^N, U_s: s\le t)$,
	and define the space of all (fully observable) admissible controls by 
	\begin{equation*}
		\mathcal{A}_N := \{\alpha = (\alpha_t)_{t\le T}: \alpha \text{ is an }A\text{-valued } \mathbb{F}^{N}\text{-predictable process}\}.
	\end{equation*}
	Given a control process $\alpha \in \Ac_N$, the controlled particle system $(X_t^{1,\alpha}, \dots, X_t^{N,\alpha}, L_t^{1,\alpha}, \dots, L_t^{N,\alpha})$ is defined by,
	for each $k = 1, \cdots, N$,
	\begin{align} 
		X_t^{k,\alpha} &=~ \xi_k + \int_0^t \Big(b(s, X_s^{k,\alpha},  \alpha_s) - \sigma_0(s,X_s^{k,\alpha}, \alpha_s) h(s, X_s^{k,\alpha}, \alpha_s) \Big) \,\dd s \nonumber \\
		&~~~~~+ \int_0^t \sigma(s,X_s^{k,\alpha}, \alpha_s)\,\dd W_s^{k} + \int_0^t \sigma_0(s, X_s^{k, \alpha}, \alpha_s)\,\dd U_s, \label{eq:SDE_XN} \\
		L_t^{k,\alpha} &= 1 + \int_0^t  h(s, X_s^{k,\alpha}, \alpha_s) L_s^{k,\alpha} \,\dd U_s, \label{eq:SDE_LN}
	\end{align}
	
	Under Assumptions \ref{assum:coeff}, it is clear that SDE \eqref{eq:SDE_XN}-\eqref{eq:SDE_LN} has a unique strong solution.
	Our continuous-time particle approximation problem is then given by 
	\begin{equation} \label{eq:continuous-N-particle-problem}
		V_0^N :=\inf_{\alpha\in \mathcal{A}_N} J^N(\alpha), 
	\end{equation}
	with
	\begin{equation} \label{eq:defJN}
		J^N(\alpha)
		:=
		\mathbb{E}^{\mathbb{Q}^N}\left[ \int_0^T \frac{1}{N}\sum_{k=1}^N L_t^{k,\alpha} f(t, X_t^{k,\alpha}, \alpha_t)\,\dd t + \frac{1}{N}\sum_{k=1}^N L_T^{k,\alpha} g(X_T^{k,\alpha}) \right].
	\end{equation}
	
	We then have the following convergence result.
	\begin{proposition} \label{prop:CVG_VN}
		Let Assumptions \ref{assum:coeff} hold true with $p > q = 2$, and assume that $\sigma_0(t,x,a)$ and $h(t,x,a)$ are independent of the control variable $a \in A$. 
		Then
		$$
		V_0^N \longrightarrow V_0,
		~~\mbox{as}~~
		N \longrightarrow \infty.
		$$
	\end{proposition}
	\begin{proof}
		The reformulation \eqref{eq:SDE_X}-\eqref{eq:change-of-measure}-\eqref{eq:reform_V0} of the partially observed control problem under the probability $\mathbb{Q}$ falls in the setting of the common noise adapted mean-field control problem (11) of \cite{bouchard2025limit}.
		In particular, with the compactness condition on $A$ and the Lipschitz conditions on the coefficient functions in Assumption \ref{assum:coeff},
		the required convergence result follows directly by \cite[Theorem 3.2]{bouchard2025limit}.
	\end{proof}

	\paragraph{BSDE characterization of the particle system when $\sigma$ is not controlled.}
	
	The $N$-particle problem~\eqref{eq:continuous-N-particle-problem} is in the form of standard optimal control problems.
	When the coefficient $\sigma$ and $h$ are uncontrolled and non-degenerate (see Assumption \ref{assum:BSDE} for a precise statement), one can obtain a Backward SDE (BSDE) characterization of the value function by following the standard approach, see e.g. \cite[Theorem 4.5.3]{zhang2017backward}.
	
	\vspace{0.5em}
	
	First, given the arguments $a \in A$, $\mathbf{x} = (x^1, \cdots, x^N)$ and $\bm{\ell} = (\ell^1, \cdots, \ell^N)$, 
	let denote $b_t^k := b(t, x^k, a)$, $\sigma_t^k := \sigma(t, x^k)$, $\sigma_{0,t}^k := \sigma_0(t, x^k)$, and $h_t^k := h(t, x^k)$. Assume the $\sigma$ is invertible. One can then identify the following equality for the  drift and volatility of the $N$-particle system: 
	\begin{equation}
		\begin{pmatrix} b_t^1 - \sigma_{0,t}^1 h_t^1 \\ \vdots \\ b_t^N - \sigma_{0,t}^N h_t^N 
			\\ 0 \\ \vdots \\ 0 \end{pmatrix} 
		= 
		\begin{pmatrix} 
			\sigma_t^1 & & 0_{d\times d} & \sigma_{0,t}^1 \\ 
			& \ddots & & \vdots \\ 
			0_{d\times d} & & \sigma_t^N & \sigma_{0,t}^N \\ 
			0_{1\times d} & \cdots & 0_{1\times d} & h_t^1 \ell^1 \\ 
			& \vdots & & \vdots \\ 
			0_{1\times d} & \cdots & 0_{1\times d} & h_t^N \ell^N 
		\end{pmatrix} 
		\begin{pmatrix} 
			(\sigma_t^1)^{-1}(b_t^1 - \sigma_{0,t}^1 h_t^1) \\ 
			\vdots \\ 
			(\sigma_t^N)^{-1}(b_t^N - \sigma_{0,t}^N h_t^N) \\ 
			0_d 
		\end{pmatrix}.
	\end{equation}
	We then define the Hamiltonian $\mathcal{H} : [0,T] \x \R^{d \x N} \x \R^{N} \x \R^{d \x (N+1)} \longrightarrow \R$ by 
	\begin{equation*}
		\mathcal{H}(t,\mathbf{x},\bm{\ell},\mathbf{z}) 
		~:=~
		\inf_{a \in A}  ~ \bigg( \frac1N \sum_{k=1}^N \ell^k f(t, x^k, a) + \mathbf{z}\cdot \boldsymbol{\theta}(t,\mathbf{x},\bm{\ell}, a) \bigg),
	\end{equation*}
	where
	\begin{equation*}
		\boldsymbol{\theta}(t,\mathbf{x}, \bm{\ell}, a) = \begin{pmatrix} 
			(\sigma_t^1)^{-1}(b_t^1 - \sigma_{0,t}^1 h_t^1) \\ 
			\vdots \\ 
			(\sigma_t^N)^{-1}(b_t^N - \sigma_{0,t}^N h_t^N) \\ 
			0_d 
		\end{pmatrix}.
	\end{equation*}
	Next, Proposition~\ref{prop:bsde} provides the BSDE representation of the optimal value for the $N$-particle problem~\eqref{eq:continuous-N-particle-problem} provided that the diffusion $\sigma$ is uncontrolled and invertible, whose proof can be found in \cite[Theorem 4.5.3]{zhang2017backward}. 
	
	\begin{assumption} \label{assum:BSDE}
		The coefficient function $\sigma$, $\sigma_0$ and $h$ are uncontrolled, so that one can write $(\sigma, \sigma_0, h)(t,x)$ in place of $(\sigma, \sigma_0, h)(t,x,a)$.
		Moreover, the function $\mathbf{\theta}$ is uniformly bounded so that the Hamiltonian $\mathcal{H}(t,\mathbf{x},\bm{\ell},\mathbf{z})$ is Lipschitz in $\mathbf{z}$.
	\end{assumption}
	
	Let us define, in the filtered probability space $(\Omega^N, \Fc^N, \F^N, \bP^N)$ equipped with independent Brownian motions $W^1, \dots, W^N, U$,
	the uncontrolled process $\mathbf{X}_t = (X_t^1, \dots, X_t^N)$, $\mathbf{L}_t = (L_t^1, \dots, L_t^N)$ by the following SDEs: for $k=1, \cdots, N$,
	$$
	X_t^k = \xi_k +  \int_0^t \sigma(s, X_s^k)\,\dd W_s^k + \int_0^t \sigma_0(s, X_s^k)\,\dd U_s, \quad t \in [0,T],
	$$
	and
	$$
	L_t^k = 1 + \int_0^t  h(s, X_s^k)L_s^k\,\dd U_s, \quad t \in [0,T].
	$$
	Let us denote by $\mathbb{S}^{N,2}$ the space of all $\F^N$-optional $\R$-valued process $Y = (Y_t)_{0 \le t \le T}$ such that $\bE \big[ \sup_{0 \le t \le T} |Y_t|^2 \big] < \infty$,
	and by $\mathbb{H}^{N,2}$ the space of all $\F^N$-predictable $\R^{d \times (N+1)}$-valued process $Z = (Z_t)_{0 \le t \le T}$ such that
	$$
	\bE \Big[ \int_0^T \big| Z_t \big|^2 dt \Big] < \infty.
	$$
	
	\begin{proposition} \label{prop:bsde}
		Let Assumptions \ref{assum:coeff} and \ref{assum:BSDE} hold true with parameters $p > 2 q \ge 2$.
		Then the following BSDE has a unique solution $(Y,  \mathbf{Z}_t = (Z_t^1,\dots,Z_t^N, Z_t^{N+1})) \in \mathbb{S}^{N,2} \x \mathbb{H}^{N,2}$:
		\begin{equation} \label{eq:FBSDE}
			\begin{aligned}
				\dd Y_t^N &=  - \mathcal{H}(t, \mathbf{X}_t, \mathbf{L}_t, \mathbf{Z}_t)\,\dd t + \sum_{k=1}^N Z_t^k \,\dd W_t^k + Z_t^{N+1} \,\dd U_t, 
				\quad Y_T^N = \frac{1}{N}\sum_{k=1}^N L_T^k g(X_T^k).
			\end{aligned}
		\end{equation}
		and $Y_t$ provides the value process for the $N$-particle problem~\eqref{eq:continuous-N-particle-problem}, in particular, $Y_0 = V^N_0$.
	\end{proposition}
	\begin{proof}
		When Assumption \ref{assum:coeff} holds with $p > 2 q \ge 2$, it is direct to check that 
		$$
		\bE \Big[ 
		\int_0^T \big| \mathcal{H}(t, \mathbf{X}_t, \mathbf{L}_t, \mathbf{0} ) \big|^2 \,\dd t 
		+
		\sum_{k=1}^N \big| L_T^k g(X_T^k) \big|^2
		\Big] < \infty.
		$$
		Further, under Assumption \ref{assum:BSDE}, the Hamiltonian $\mathcal{H}(t,\mathbf{x},\bm{\ell},\mathbf{z})$ is Lipschitz in $\mathbf{z}$.
		Then it is standard that the BSDE \eqref{eq:FBSDE} has a unique solution in $\mathbb{S}^{N,2} \x \mathbb{H}^{N,2}$ 
		(see e.g. \cite[Theorem 4.3.1]{zhang2017backward}).
		Finally, the fact that $Y_t$ coincides with the value function of the control problem \eqref{eq:continuous-N-particle-problem} follows by \cite[Theorem 4.5.3]{zhang2017backward}.
	\end{proof}
	
	\paragraph{A discrete-time controlled particle system.}
	
	The numerical implementation of the approximation method relies on the time discretization.
	We therefore consider the following discrete-time particle system associated with the continuous-time problem~\eqref{eq:continuous-N-particle-problem}. 
	Let $\mathcal{T}:=\{t_i=i\Delta t: \Delta t = \frac{T}{N_T}, i=0,1,\dots, N_T\}$ be a discrete time grid, we define the discrete-time processes $(X^k, L^k)_{k = 1, \cdots, N}$ on $\mathcal{T}$ by Euler-Maruyama scheme of SDE \eqref{eq:SDE_XN}-\eqref{eq:SDE_LN}: 
	given a control process $(\alpha_i)_{0 \le i \le N_T-1}$ such that $\alpha_i \in \Fc^N_{t_i}$, we denote $\Delta W^k_{i+1} := W^k_{t_{i+1}} - W^k_{t_i}$ and $\Delta U_{i+1} := U_{t_{i+1}} - U_{t_i}$,
	and then define
	\begin{align} \label{eq:SDE_X_euler}
		X^{k,\Delta t}_{t_0} := \xi_k,
		\quad
		X_{t_{i+1}}^{k,\Delta t} &= X_{t_i}^{k,\Delta t} + \left(b(t_i, X_{t_i}^{k,\Delta t}, \alpha_i) - \sigma_0(t_i,X_{t_i}^{k,\Delta t}, \alpha_i)h(t_i,X_{t_i}^{k,\Delta t}, \alpha_i) \right)\Delta t  \nonumber \\
		&\qquad  + \sigma(t_i, X_{t_i}^{k,\Delta t}, \alpha_i)\Delta W_{i+1}^k + \sigma_0(t_i,X_{t_i}^{k,\Delta t}, \alpha_i)\Delta U_{i+1},
	\end{align}
	and
	\begin{equation} \label{eq:SDE_L_euler}
		L^{k,\Delta t}_{t_0} = 1,
		\quad 
		L_{t_{i+1}}^{k,\Delta t} = L_{t_i}^{k,\Delta t}\exp\left( -\frac{1}{2}h(t_i, X_{t_i}^{k,\Delta t}, \alpha_i)^2\Delta t + h(t_i, X_{t_i}^{k,\Delta t}, \alpha_i)\Delta U_{i+1} \right),
		~~i=0, \cdots, N_T-1.
	\end{equation}
	Let us denote by $\overline{\mathcal{A}}_{N, \Delta t}$ the space of all such control processes. 
	Then given $\alpha = (\alpha_0, \cdots, \alpha_{N_T-1}) \in \overline{\mathcal{A}}_{N, \Delta t}$, the discrete-time cost function associated with~\eqref{eq:defJN} is given by 
	\begin{equation}\label{eq:discrete-time-prob}
		J^{N,\Delta t}(\alpha)
		:= 
		\mathbb{E}\left[ \sum_{i=0}^{N_T-1} \frac{1}{N}\sum_{k=1}^N L_{t_i}^{k,\Delta t} f(t_i, X_{t_i}^{k,\Delta t},  \alpha_i) \Delta t +\frac{1}{N}\sum_{k=1}^N L_T^{k,\Delta t} g(X_T^{k,\Delta t}) \right], 
	\end{equation}
	and the optimal value of the discrete-time controlled particle system is defined by 
	\begin{equation} \label{eq:def_V0NDt}
		V_0^{N,\Delta t} 
		~:=
		\inf_{\alpha \in \overline{\mathcal{A}}_{N, \Delta t}} J^{N, \Delta t}(\alpha). 
	\end{equation}
	
	\begin{remark}
		We apply the Euler-Maruyama scheme to the SDE satisfied by $\log(L_t^k)$ to ensure the positivity of $L_{t_i}^k$, and thus the positivity of the weight $w_{t_i}^k$ in the mean-field setting, see~\eqref{eq:def_muN_MF} below. A direct discretization of the dynamics of $L_t^k$ can produce negative values of $L_{t_{i+1}}^k$ when the Gaussian noise $\Delta U_i$ is large. 
	\end{remark}
	
	\begin{remark} \label{rem:feedback-control}
		The Problem~\eqref{eq:def_V0NDt} is time-consistent and thus can be solved by dynamic programming. The optimal control then satisfies a feedback form $\alpha_i = a_i(X_{t_i}^1,\dots, X_{t_i}^N, L_{t_i}^1,\dots, L_{t_i}^N)$. We shall approximate the feedback policy $a_i$ by neural networks in numerical implementations. 
	\end{remark}
	
	The following result shows the convergence of $V_0^{N,\Delta t}$ to the optimal value $V_0$ of the partially observable control problem~\eqref{eq:original-prob} 
	as $N \longrightarrow \infty$ and $\Delta t \longrightarrow 0$. 
	\begin{theorem} \label{thm:cvg_discrete}
		Let Assumptions \ref{assum:coeff} hold true with $p > q = 2$, 
		and assume that $\sigma_0(t,x,a)$ and $h(t,x,a)$ are independent of the control variable $a \in A$.
		Suppose in addition that $f(t,x,a)$ and $g(x)$ are uniformly continuous in $(t,x)$ with modulus of continuity $\varpi(\cdot)$.
		Then
		\begin{equation*}
			V_0^{N, \Delta t} \longrightarrow V_0,  ~\text{ as }N\longrightarrow \infty,~ \Delta t \longrightarrow 0. 
		\end{equation*}
	\end{theorem}
	\begin{proof}
		First, a discrete-time control $\alpha = (\alpha_0, \cdots, \alpha_{N_T-1}) \in \overline{\mathcal{A}}_{N, \Delta t}$ can be considered as a piecewise constant continuous-time control process by extending the definition:
		$$
		\alpha_t := \alpha_{t_i}, ~t \in [t_i, t_{i+1}], ~~i = 0, 1, \cdots, N_T-1.
		$$
		Similarly, one can consider $X^{k,\Delta t}$ as a continuous-time process by similar extension of definitions.
		Further, let us denote by $\mathcal{A}_0$ the space  of all  piecewise constant control processes $\alpha \in \mathcal{A}$
		then it follows by \cite[Proposition A.5]{bouchard2025limit} that 
		$$
		V_0 = \inf_{\alpha \in \mathcal{A}_0} \mathbb{E}^{\mathbb{Q}} 
		\left[ 
		\int_0^T L_t ~f(t, X_t,  \alpha_t)\,\dd t 
		+
		L_T~ g(X_T) 
		\right]. 
		$$
		
		In view of this, by considering the discrete-time control $\alpha \in \overline{\mathcal{A}}_{N, \Delta t}$ as piecewise constant continuous-time control and recalling the definition of $J^N$ in \eqref{eq:defJN},
		one can deduce by exactly the same arguments as in the proof of \cite[Theorem 3.2]{bouchard2025limit} that
		\begin{equation} \label{eq:VNDt2V}
			\inf_{\alpha \in \overline{\mathcal{A}}_{N, \Delta t}} J^N(\alpha) \longrightarrow V_0,
			~\text{ as }N\longrightarrow \infty,~ \Delta t \longrightarrow 0. 
		\end{equation}
		
		At the same time, \eqref{eq:SDE_X_euler}-\eqref{eq:SDE_L_euler} is the Euler scheme for SDE \eqref{eq:SDE_XN}-\eqref{eq:SDE_LN},
		one has by standard analysis (see e.g. \cite{graham2013stochastic}) that, there exists a constant $C> 0$ such that,
		for all $N_T \ge 1$, $\Delta t:= T/N_T$ and all $\alpha \in \overline{\mathcal{A}}_{N, \Delta t}$
		$$
		\bE \Big[ \sup_{0 \le t \le T} \left(\big| X^{k,\Delta t}_t - X^k_t \big|^2 + \big|L_t^{k,\Delta t} - L_t^k\big|^2 \right) \Big] \le C \Delta t.
		$$
		Together with uniform continuity condition on $f(\cdot)$ and $g(\cdot)$, it follows that
		$$
		\sup_{\alpha \in \overline{\mathcal{A}}_{N, \Delta t}} \big| J^N(\alpha) - J^{N, \Delta t} (\alpha) \big| 
		~\le~
		C \varpi(\sqrt{\Delta t}).
		$$
		Together with \eqref{eq:VNDt2V} and \eqref{eq:def_V0NDt}, it follows that
		$$
		V_0^{N, \Delta t} \longrightarrow V_0,  ~\text{ as }N\longrightarrow \infty,~ \Delta t \longrightarrow 0. 
		$$
	\end{proof}

	\subsection{An extension to the mean-field optimal control with partial observation} \label{subsec:extension-mfc}
	
	The above particle approximation method can be naturally extended for the mean-field optimal control problem with partial observation, in the spirit of  \cite{buckdahn2017mean}.
	Recall that $\Xc = \mathcal{U} = \R^d$, we denote by $\Pc(\Xc)$  the space of all Borel probability measures on $\Xc$.
	In the mean-field setting, we are given the coefficient functions $(b, \sigma, \sigma_0, f): [0, T] \times \mathcal{X} \times \mathcal{P}(\mathcal{X}) \times A \to \mathbb{R}^d\times \mathbb{R}^{d\times d}\times \mathbb{R}^{d\times d} \times \R$, $h:[0,T]\times \mathcal{X} \times A \to \mathcal{U}$, and $g:\mathcal{X}\times \mathcal{P}(\mathcal{X})\to \R$,
	together with the initial distribution $\mu_0 \in \mathcal{P}(\Xc)$.
	
	\paragraph{A mean-field control problem with partial observation.}
	In the mean-field setting, the unobservable state process $X$ is defined by the McKean--Vlasov SDE
	\begin{equation}
		X_t 
		= 
		X_0
		+
		\int_0^t  b(s, X_s, \mu_s,  \alpha_s)\,\dd s 
		+
		\int_0^t \sigma(s, X_s, \mu_s,  \alpha_s) \dd W_s 
		+
		\int_0^t \sigma_0(s, X_s, \mu_s,  \alpha_s)\dd B_s,
	\end{equation}
	where $X_0 \sim \mu_0$, $W$ and $B$ are two independent $\R^d$-valued standard Brownian motions in the probability space $(\Omega, \Fc, \bP)$.
	The observable process $U$ is defined by
	\begin{equation}
		U_t 
		=
		\int_0^t h(s, X_s, \alpha_s)\,\dd s 
		+ 
		B_t,
		~
		t \in [0,T]. 
	\end{equation}
	Finally, a control process $\alpha$ is predictable w.r.t. the sub-filtration  $\G = (\Gc_t)_{0 \le t \le T}$ generated by $U$,
	and
	$$
	\mu_t = \Lc^{\bP} \big( X_t \big| \Gc_t \big),
	~~t \in [0,T].
	$$
	
	Similarly, we can follow the classical reference probability approach to define the equivalent probability measure $\mathbb{Q}$  by 
	\begin{equation} \label{eq:change-of-measure_MF}
		\frac{\dd \mathbb{Q}}{\dd \mathbb{P}}
		:= \big(L_T \big)^{-1}, 
		~\mbox{with}~
		L_t:= \exp\left( \int_0^t h(s,X_s, \alpha_s)\,\dd U_s - \frac{1}{2}\int_0^t \big| h(s, X_s, \alpha_s) \big|^2\,\dd s\right),
		~t \in [0,T],
	\end{equation}
	so that  $W$ and $U$ are two independent standard Brownian motions under $\mathbb{Q}$,
	the processes $X$ satisfy $\mathcal{L}^{\Q} (X_0) = \mu_0$ and
	\begin{align} \label{eq:SDE_X_MF}
		X_t 
		=&~
		X_0 
		+
		\int_0^t \big(b(s, X_s, \mu_s, \alpha_s) - \sigma_0(s, X_s, \mu_s, \alpha_s) h(s,X_s, \alpha_s) \big) \dd s \nonumber \\
		& +
		\int_0^t \sigma(s, X_s,  \mu_s, \alpha_s)\dd W_s 
		+ 
		\int_0^t \sigma_0(s, X_s, \mu_s, \alpha_s)\dd U_s, 
		~~t \in [0,T],
	\end{align}
	and $L$ is defined in \eqref{eq:change-of-measure_MF}, $\mu_t$ satisfies, for any bounded continuous function $\varphi \in C_b( \R^d)$,
	\begin{equation} \label{eq:def_mu_MF}
		\big\langle \mu_t, \varphi \big \rangle 
		~=~
		\frac{\mathbb{E}^{\mathbb{Q}} \big [L_t \varphi(X_t ) \big| \Gc_t \big]}{\mathbb{E}^{\mathbb{Q}} \big[L_t \big| \Gc_t \big]},
		~~t \in [0,T].
	\end{equation}
	We then obtain a general mean-field control problem with partial observation:
	\begin{equation} \label{eq:reform_V0_MF}
		V_0 
		:=
		\inf_{\alpha \in \mathcal{A}} 
		\mathbb{E}^{\mathbb{Q}} 
		\left[ 
		\int_0^T L_t ~f(t, X_t, \mu_t, \alpha_t)\,\dd t 
		+
		L_T g(X_T, \mu_T) 
		\right]. 
	\end{equation}
	
	\begin{remark}
		Let us refer to \cite{buckdahn2023general} for an existence and uniqueness result of the mean-field SDE \eqref{eq:change-of-measure_MF}-\eqref{eq:SDE_X_MF} with interaction term $\mu_t$ defined in \eqref{eq:def_mu_MF}.
	\end{remark}
	
	\paragraph{Particle approximation of mean-field control problem with partial observation.}
	
	For the uncontrolled mean-field SDE \eqref{eq:change-of-measure_MF}-\eqref{eq:def_mu_MF}-\eqref{eq:SDE_X_MF} as studied in \cite{buckdahn2023general},
	a particle approximation method has been studied by \cite{du2024particle} with an explicit convergence rate.
	Combining with approach in Section \ref{subsec:particle-approx}, one can naturally obtain a particle approximation method for the mean-field control problem with partial observation.
	
	\vspace{0.5em}
	
	Let $N \ge 1$, $(\Omega^N, \mathcal{F}^N, \mathbb{Q}^N)$ be a probability space equipped with a Brownian motion $U$, a sequence of i.i.d. random variables $(\xi_k)_{k=1}^N$ with $\xi_k \sim \mu_0$, and a sequence of Brownian motions $(W^k)_{k=1}^N$, all of which are mutually independent. 
	Let us define the filtration $\mathbb{F}^N:=(\mathcal{F}_t^N)_{t\in [0,T]}$ by $\mathcal{F}_t^N:=\sigma(\xi_1,\dots, \xi_N, W_s^1,\dots, W_s^N, U_s: s\le t)$,
	and define the space of all (fully observable) admissible controls by 
	\begin{equation*}
		\mathcal{A}_N := \{\alpha = (\alpha_t)_{t\le T}: \alpha \text{ is an }A\text{-valued } \mathbb{F}^{N}\text{-predictable process}\}.
	\end{equation*}
	Given a control process $\alpha \in \Ac_N$, the controlled particle system $(X_t^{1,\alpha}, \dots, X_t^{N,\alpha}, L_t^{1,\alpha}, \dots, L_t^{N,\alpha})$ is defined by,
	for each $k = 1, \cdots, N$,
	\begin{align}
		X_t^{k,\alpha} &=~ \xi_k + \int_0^t \Big(b(s, X_s^{k,\alpha}, \mu_s^{N,\alpha}, \alpha_s) - \sigma_0(s,X_s^{k,\alpha}, \mu_s^{N,\alpha}, \alpha_s) h(s, X_s^{k,\alpha}, \alpha_s) \Big) \,\dd s \nonumber \\
		&~~~~~+ \int_0^t \sigma(s,X_s^{k,\alpha}, \mu_s^{N,\alpha}, \alpha_s)\,\dd W_s^{k} + \int_0^t \sigma_0(s, X_s^{k, \alpha}, \mu_s^{N, \alpha}, \alpha_s)\,\dd U_s, \label{eq:SDE_XN_MF} \\
		L_t^{k,\alpha} &= 1 + \int_0^t  h(s, X_s^{k,\alpha}, \alpha_s) L_s^{k,\alpha} \,\dd U_s, \label{eq:SDE_LN_MF}
	\end{align}
	where
	\begin{equation} \label{eq:def_muN_MF}
		\mu_t^{N,\alpha} := \frac{\sum_{k=1}^N L_t^{k,\alpha} \delta_{X_t^{k,\alpha}}}{\sum_{j=1}^N L_t^{j,\alpha}} = \sum_{k=1}^N w_t^{k,\alpha} \delta_{X_t^{k,\alpha}},
		~~\text{ with the weight }w_t^{k,\alpha}:= \frac{L_t^{k,\alpha}}{\sum_{j=1}^N L_t^{j,\alpha}}. 
	\end{equation}
	
	A continuous-time particle approximation problem is then given by 
	\begin{equation} \label{eq:continuous-N-particle-problem_MF}
		V_0^N :=\inf_{\alpha\in \mathcal{A}_N} J^N(\alpha), 
	\end{equation}
	with
	$$
	J^N(\alpha)
	:=
	\mathbb{E}^{\mathbb{Q}^N}\left[ \int_0^T \frac{1}{N}\sum_{k=1}^N L_t^{k,\alpha} f(t, X_t^{k,\alpha}, \mu_t^{N,\alpha},\alpha_t)\,\dd t + \frac{1}{N}\sum_{k=1}^N L_T^{k,\alpha} g(X_T^{k,\alpha}, \mu_T^{N,\alpha}) \right].
	$$
	
	Unlike Proposition \ref{prop:CVG_VN}, one cannot directly apply the results in \cite{bouchard2025limit} to deduce a convergence result, 
	which we would like to study in our future work.

	\section{Numerical schemes} \label{sec:numerical-schemes}
	
	In this section, we develop numerical schemes for the $N$-particle problem~\eqref{eq:continuous-N-particle-problem_MF}, which includes~\eqref{eq:continuous-N-particle-problem} as a special case without mean-field interactions. The resulting problem is a time-consistent stochastic control problem in a high-dimensional state space of dimension $N(d+1)$. To address this challenge, we propose a machine-learning-based framework designed for high-dimensional stochastic control problems. 
	
	To highlight the main ideas, we illustrate the numerical method in the one-dimensional setting ($d=1$) in this section; the extension to general $d$ dimensions is straightforward, and corresponding numerical experiments are provided in Section~\ref{sec:numerical-experiments}. We employ deep neural networks to solve the resulting $2N$-dimensional stochastic control problem. Moreover, since the problem is invariant under permutations of the particle states $(X^1, L^1), \dots, (X^N, L^N)$, we exploit this inherent symmetry to design suitable neural network approximations. 
	
	\paragraph{Neural network architectures.}
	
	Deep learning-based methods approximate the unknown quantities of the problem, such as the value function, feedback control policies, and the $Z$ component in the solution of the associated BSDEs, using neural networks.  Let us denote the collection of single-layer neural networks with activation function $\sigma$ by 
	\begin{equation*}
		\mathcal{L}_{d_1, d_2}^\sigma := \big\{ \phi:\R^{d_1} \to \R^{d_2}\mid \phi(x) = \sigma(Wx +b), W\in \R^{d_2\times d_1}, b\in \R^{d_2} \big\}. 
	\end{equation*}
	For the identity activation $\sigma(x)=x$, we write $\mathcal{L}_{d_1, d_2}^{\mathrm{id}}$ to represent the collection of affine transformations. A multi-layer feedforward neural network (FNN) is constructed via the composition of these single layers. We define the class of FNNs as 
	\begin{equation*}
		\begin{aligned}
			\mathcal{N}_{\ell, d_0, d_h, d_{\ell+1}}^\sigma := \big\{\Phi:\R^{d_0} \to \R^{d_{\ell+1}}\mid \Phi = \phi_\ell \circ \phi_{\ell-1} \circ \dots \circ \phi_0 \text{ with } &\phi_0 \in \mathcal{L}_{d_0,d_h}^\sigma, \\
			&\phi_{i} \in \mathcal{L}_{d_h,d_h}^\sigma, i=1,\dots, \ell-1, \\
			&\phi_\ell \in \mathcal{L}_{d_h, d_{\ell+1}}^{\mathrm{id}}\big\}.
		\end{aligned}
	\end{equation*}
	Here, $\ell \in \mathbb{N}$ is the number of hidden layers, while $d_0, d_h$, and $d_{\ell+1}$ denote the dimensions of input, hidden, and output layers, respectively. Since any $\Phi \in \mathcal{N}_{\ell, d_0, d_h, d_{\ell+1}}^\sigma$ is uniquely determined by its weights and biases, we explicitly denote this dependence by writing $\Phi = \Phi(x;\theta)$ where $\theta = (W_0, b_0, \dots, W_\ell, b_{\ell})$ collects all trainable parameters across the network. 
	
	Building on FNNs, for a generic input $x = (x^1, \dots, x^N)$, the symmetric neural network architecture tailored to the $N$-particle problem is defined as 
	\begin{equation} \label{eq:symmetric-nn}
		v(x;\theta) := \Phi_2\left(\frac{1}{N}\sum_{k=1}^N \Phi_1(x^k;\theta_1);\theta_2\right), \quad \theta = (\theta_1, \theta_2), 
	\end{equation}
	where the inner network $\Phi_1\in \mathcal{N}_{\ell, d_0, d_h, d_m}^\sigma$ maps individual states to a latent space of dimension $d_m\in \mathbb{N}$, and the outer network $\Phi_2 \in \mathcal{N}_{\ell, d_m, d_h, d_{out}}^\sigma$ maps the aggregated latent representation to the final output space of dimension $d_{out}\in \mathbb{N}$. This specific symmetric architecture, widely known as DeepSets \cite{germain2022deepsets}, has also been used to approximate population distributions in \cite{dayanikli2023deep}. Because the input to the outer network $\Phi_2$ is the empirical mean of the individual feature representations $\Phi_1(x^k; \theta_1)$, the resulting function $v(x;\theta)$ is invariant under any permutation of the input sequence $(x^1, \dots, x^N)$. 
	
	\subsection{A direct approach for the \texorpdfstring{$N$}{}-particle problem} \label{subsec:direct-approach}
	
	We present a direct approach to the numerical solution of the discrete-time problem associated with~\eqref{eq:continuous-N-particle-problem_MF}, including~\eqref{eq:def_V0NDt} as the special case without mean-field interactions. Although our theoretical framework applies to the partially observable problem without mean-field interactions, we illustrate the algorithm here in the presence of the distribution term and will investigate its numerical performance in Section~\ref{sec:numerical-experiments}. The direct approach restricts the feedback control policy to a class of neural networks whose parameters are learned by stochastic optimization of the expected total cost. The direct method for stochastic control problems has been developed in \cite{han2016deep} and extended to the mean-field problems in \cite{carmona2022convergence}.
	
	The optimal control $\alpha_i$ in the discrete-time $N$-particle problem~\eqref{eq:def_V0NDt} admits a feedback form (see Remark~\ref{rem:feedback-control}) and is invariant under permutations of $(X^k, L^k)_{k=1}^N$. We therefore parametrize the control sequence $(\alpha_i)_{i=0}^{N_T-1}$ by a sequence of symmetric neural networks $(a_i(\cdot;\theta^i))_{i=0}^{N_T-1}$ with trainable parameters $\theta = (\theta^0,\dots, \theta^{N_T-1})$, such that 
	$$\alpha_i \approx \alpha_{\theta^i} := a_i\left((X_{t_i}^1, L_{t_i}^1),\dots, (X_{t_i}^N, L_{t_i}^N); \theta^i\right).$$ 
	The discrete-time control problem associated with \eqref{eq:continuous-N-particle-problem_MF} is then approximated by minimizing, over $\theta$, the expected loss function 
	\begin{equation*}
		\mathcal{L}(\theta) := \mathbb{E}\left[ \sum_{i=0}^{N_T-1} \frac{1}{N}\sum_{k=1}^N L_{t_i}^{k} f\left(t_i, X_{t_i}^{k}, \mu_{t_i}^{N}, \alpha_{\theta^i}\right) \Delta t + \frac{1}{N}\sum_{k=1}^N L_T^{k} g(X_T^{k}, \mu_T^{N}) \right], 
	\end{equation*}
	under the dynamic constraint: 
	\begin{equation} \label{eq:sample-trajectory}
		\begin{aligned}
			X_{t_0}^k = \xi_k, \quad X_{t_{i+1}}^k &= X_{t_i}^k + \left(b(t_i, X_{t_i}^k, \mu_{t_i}^N, \alpha_{\theta^i}) - \sigma_0(t_i, X_{t_i}^k, \mu_{t_i}^N, \alpha_{\theta^i})h(t_i,X_{t_i}^k,\alpha_{\theta^i}) \right) \\
			&\quad + \sigma(t_i, X_{t_i}^k, \mu_{t_i}^N, \alpha_{\theta^i}) \Delta W_{i+1}^k + \sigma_0(t_i, X_{t_i}^k, \mu_{t_i}^N, \alpha_{\theta^i})\Delta U_{i+1}, \\
			L_{t_0}^k = 1,\quad  L_{t_{i+1}}^k &= L_{t_i}^k\exp\left( -\frac{1}{2}h(t_i, X_{t_i}^k, \alpha_{\theta^i})^2\Delta t + h(t_i, X_{t_i}^k, \alpha_{\theta^i})\Delta U_{i+1} \right), 
		\end{aligned}
	\end{equation}
	where $\mu_{t_i}^N = \sum_{k=1}^N w_{t_i}^k \delta_{X_{t_i}^k}$ with $w_{t_i}^k = \frac{L_{t_i}^k}{\sum_{j=1}^N L_{t_i}^j}$, and $(\Delta W_{i}^1, \dots, \Delta W_i^N, \Delta U_i)_{i=1}^{N_T}$ are i.i.d. Gaussian with distribution $\mathcal{N}(0,\Delta t)$. 
	
	\paragraph{Training method.}
	
	Since the loss function is written as an expectation, one can rely on stochastic optimization algorithms. We use the Adam optimizer to train the control networks $(a_i(\cdot; \theta^i))_{i=0}^{N_T-1}$. The randomness inside the expectation depends on a collection of random variables 
	\begin{equation} \label{eq:randomness-S}
		S = \left((\xi_k)_{k=1,\dots, N}, (\Delta W_i^k)_{i=1, \dots, N_T, k=1,\dots, N},(\Delta U_i)_{i=1,\dots,N_T}\right).
	\end{equation}
	Hence, given a batch of $M$ realizations of $S$ and a choice of trainable parameter $\theta$, the trajectory of $(X^{k, (m)}_{t_i}, L^{k, (m)}_{t_i})_{i=0, \dots, N_T, k=1,\dots, N}$ can be sampled by \eqref{eq:sample-trajectory} for each $m \in \{1,\dots, M\}$, and the expected loss is then approximated by 
	\begin{equation}\label{eq:batch-loss}
		\begin{aligned} 
			\mathcal{L}^{M,N}(\theta) = \frac{1}{M}\sum_{m=1}^M \bigg[ \sum_{i=0}^{N_T-1} \frac{1}{N}\sum_{k=1}^N L_{t_i}^{k, (m)} f\left(t_i, X_{t_i}^{k, (m)}, \mu_{t_i}^{N, (m)}, \alpha_{\theta^i}^{(m)}\right) \Delta t \\
			+ \frac{1}{N}\sum_{k=1}^N L_T^{k, (m)} g(X_T^{k, (m)}, \mu_T^{N, (m)}) \bigg], 
		\end{aligned}
	\end{equation}
	where we denote $\alpha_{\theta^i}^{(m)} = a_i\left( (X_{t_i}^{1, (m)}, L_{t_i}^{1, (m)}), \dots, (X_{t_i}^{N, (m)}, L_{t_i}^{N, (m)}); \theta^i\right)$ and $\mu_{t_i}^{N, (m)} = \sum_{k=1}^N w_{t_i}^{k,(m)} \delta_{X_{t_i}^{k, (m)}}$ with $w_{t_i}^{k, (m)} = \frac{L_{t_i}^{k, (m)}}{\sum_{j=1}^N L_{t_i}^{j, (m)}}$. 
	
	The direct approach using Adam is described in Algorithm~\ref{alg:direct-approach}. 
	\begin{algorithm}[ht]
		\caption{Direct approach using Adam}
		\label{alg:direct-approach}
		\begin{algorithmic}[1]
			\REQUIRE Initial parameter $\theta^{(0)}$; number of epochs $K$; sequence $(\bar{\gamma}^{(k)})_{k=0,\dots, K-1}$ of learning rates; Adam hyperparameters.
			\ENSURE Optimized parameter $\theta^{(K)}$. 
			\FOR{each epoch $k$}
			\STATE Generate a batch of $M$ realizations of $S$, see \eqref{eq:randomness-S} 
			\STATE Compute sample trajectories $(X^{k, (m)}_{t_i}, L^{k, (m)}_{t_i})_{i=0, \dots, N_T, k=1,\dots, N}$ by \eqref{eq:sample-trajectory} for $m=1,\dots, M$ 
			\STATE Compute the batch loss $\mathcal{L}^{M,N}(\theta^{(k)})$ and its gradient $\nabla \mathcal{L}^{M,N}(\theta^{(k)})$, see \eqref{eq:batch-loss} 
			\STATE Set $\theta^{(k+1)}$ by one Adam update using $\nabla \mathcal{L}^{M,N}(\theta^{(k)})$ 
			\ENDFOR
			\RETURN $\theta^{(K)}$   
		\end{algorithmic}
	\end{algorithm}

	\subsection{Deep BSDE approach for the \texorpdfstring{$N$}{}-particle problem}
	When $\sigma$, $\sigma_0$, and $h$ are uncontrolled, the $N$-particle problem~\eqref{eq:continuous-N-particle-problem} (see also~\eqref{eq:continuous-N-particle-problem_MF}) admits a BSDE characterization~\eqref{eq:FBSDE}. In this section, we develop a deep BSDE approach \cite{han2018solving}, combined with symmetric neural networks~\eqref{eq:symmetric-nn}, to solve the associated FBSDE. This provides an independent method for solving the control problem under partial observations, and the resulting numerical solutions can thus be used to assess the accuracy of the direct approach. For generality, we illustrate the algorithm here in the presence of mean-field interactions. 
	
	The deep BSDE method can be interpreted as applying the direct approach to the following stochastic control problem: minimize the cost functional over $v_0:\R^{2N}\to \R$ and $z:[0,T] \times \R^{2N} \to \R^{N+1}$, given by \footnote{When $X_0$ is given as a fixed value, $V_0^{N}$ is also fixed. Thus, we set $v_0 = \theta_v$ as a trainable parameter rather than representing it with a neural network.}
	\begin{equation*}
		J_{BSDE}(v_0, z) = \mathbb{E}\left[ \bigg|V_T^{v_0, z} - \frac{1}{N}\sum_{k=1}^N L_T^k g(X_T^k, \mu^N(\mathbf{X}_T, \mathbf{L}_T)) \bigg|^2 \right], 
	\end{equation*}
	where $(V_t^{v_0, z})_{t\in [0,T]}$ acts as the controlled state satisfying 
	\begin{equation*}
		\dd V_t^{v_0, z} = -\mathcal{H}(t,\mathbf{X}_t, \mathbf{L}_t, z_t(\mathbf{X}_t, \mathbf{L}_t)) \,\dd t + z_t(\mathbf{X}_t, \mathbf{L}_t) \cdot \dd (\mathbf{W}_t, U_t), \quad V_0^{v_0, z} = v_0(\mathbf{X}_0, \mathbf{L}_0), 
	\end{equation*}
	and $(\mathbf{X}_t, \mathbf{L}_t)$ follows the uncontrolled dynamics in \eqref{eq:FBSDE}. 
	Here, $(\mathbf{W}_t, U_t)_{t\ge 0}$ is a standard $(N+1)$-dimensional Brownian motion with $\mathbf{W}_t \in \R^{N}$ and $U_t\in \R$. 
	
	Consider a uniform time grid $\mathcal{T}:=\{t_i=i\Delta t: \Delta t = \frac{T}{N_T}, i=0,1,\dots, N_T\}$. By discretizing time and parameterizing both $v_0$ and $z$ with neural networks, we obtain the deep BSDE scheme. We minimize the expected loss function over the trainable parameter $\theta = (\theta_v, \theta_z^{0},\dots,\theta_z^{N_T-1})$: 
	\begin{equation} \label{eq:loss-bsde}
		\mathcal{L}_{\text{BSDE}}(\theta) := \mathbb{E}\left[ \bigg| V_{T}^{\theta} - \frac{1}{N}\sum_{k=1}^N L_{T}^k g(X_T^k, \mu^N_T)\bigg|^2 \right], 
	\end{equation}
	subject to the following dynamic constraints: 
	\begin{equation} \label{eq:bsde-sample-trajectory}
		\begin{aligned}
			X_{t_0}^k &= \xi_k, \quad X_{t_{i+1}}^k = X_{t_i}^k + \sigma(t_i, X_{t_i}^k, \mu_{t_i}^N) \Delta W_{i+1}^k + \sigma_0(t_i, X_{t_i}^k, \mu_{t_i}^N)\Delta U_{i+1}, \\
			L_{t_0}^k &= 1, ~~\quad L_{t_{i+1}}^k = L_{t_i}^k \exp\left(-\frac{1}{2}h(t_i, X_{t_i}^k)^2 \Delta t + h(t_i, X_{t_i}^k)\Delta U_{i+1}\right), \\
			V_{t_0}^\theta &= v_0((X_{t_0}^1, L_{t_0}^1), \dots, (X_{t_0}^N, L_{t_0}^N); \theta_v), \\ 
			&\qquad \qquad  V_{t_{i+1}}^{\theta} = V_{t_i}^\theta - \mathcal{H}\left(t_i, X_{t_i}, L_{t_i}, z_{\theta^i}\right)\Delta t + z_{\theta^i} \cdot (\Delta \mathbf{W}_{i+1}, \Delta U_{i+1}), 
		\end{aligned}
	\end{equation}
	where we denote $z_{\theta^i}:= z_i\left((X_{t_i}^1, L_{t_i}^1), \dots, (X_{t_i}^N, L_{t_i}^N); \theta_z^i\right)$ using the symmetric neural network $z_i(\cdot;\theta_z^i)$ defined in \eqref{eq:symmetric-nn}, and let $\Delta \mathbf{W}_i = (\Delta W_i^1, \dotsm \Delta W_i^N)$, $\mu_{t_i}^N = \sum_{k=1}^N w_{t_i}^k \delta_{X_{t_i}^k}$ with $w_{t_i}^k = \frac{L_{t_i}^k}{\sum_{j=1}^N L_{t_i}^j}$. Here, all elements in $(\Delta W_{i}^1, \dots, \Delta W_i^N, \Delta U_i)_{i=1}^{N_T}$ are i.i.d. Gaussian with distribution $\mathcal{N}(0,\Delta t)$. 
	
	The training methodology for the deep BSDE loss~\eqref{eq:loss-bsde} closely follows the direct approach outlined in Section~\ref{subsec:direct-approach}. The randomness within the expectation of \eqref{eq:loss-bsde} stems from the collection $S$ of random variables, see \eqref{eq:randomness-S}. It is worth noting that the state trajectory of $(X_{t_i}^k, L_{t_i}^k)$ in the deep BSDE scheme differ from that in the direct approach. Specifically, given a batch of $M$ realizations of $S$ and a specific choice of parameter $\theta$, we sample $M$ trajectories $(X_{t_i}^{k, (m)}, L_{t_i}^{k,(m)}, V_{t_i}^{\theta, (m)})$ for $m=1,\dots, M$ using \eqref{eq:bsde-sample-trajectory}. The expected loss~\eqref{eq:loss-bsde} is then approximated by the empirical batch loss 
	\begin{equation} \label{eq:bsde-batch-loss}
		\mathcal{L}^{M, N}_{\text{BSDE}}(\theta) = \frac{1}{M}\sum_{m=1}^M \left[ \left| V_{T}^{\theta, (m)} - \frac{1}{N}\sum_{k=1}^N L_T^{k,(m)} g(X_T^{k, (m)}, \mu_T^{N, (m)})\right|^2 \right], 
	\end{equation}
	where $\mu_T^{N,(m)} = \sum_{k=1}^N w_T^{k, (m)}\delta_{X_T^{k,(m)}}$ with $w_T^{k,(m)} = \frac{L_T^{k,(m)}}{\sum_{j=1}^N L_T^{j,(m)}}$. 
	
	Algorithm~\ref{alg:deep-bsde} summarizes the deep BSDE approach for the $N$-particle problem~\eqref{eq:continuous-N-particle-problem_MF} and~\eqref{eq:continuous-N-particle-problem} when the diffusion $\sigma$ is uncontrolled and invertible. 
	\begin{algorithm}[ht]
		\caption{Deep BSDE approach}
		\label{alg:deep-bsde}
		\begin{algorithmic}[1]
			\REQUIRE Initial parameter $\theta^{(0)}$; number of epochs $K$; sequence $(\bar{\gamma}^{(k)})_{k=0,\dots, K-1}$ of learning rates; Adam hyperparameters.
			\ENSURE Optimized parameter $\theta^{(K)}$. 
			\FOR{each epoch $k$}
			\STATE Generate a batch of $M$ realizations of $S$, see \eqref{eq:randomness-S} 
			\STATE Compute sample trajectories $(X_{t_i}^{k, (m)}, L_{t_i}^{k,(m)}, V_{t_i}^{\theta, (m)})_{i=0, \dots, N_T, k=1,\dots, N}$ by \eqref{eq:bsde-sample-trajectory} for $m=1,\dots, M$ 
			\STATE Compute the batch loss $\mathcal{L}^{M, N}_{\text{BSDE}}(\theta^{(k)})$ and its gradient $\nabla \mathcal{L}^{M, N}_{\text{BSDE}}(\theta^{(k)})$, see \eqref{eq:bsde-batch-loss} 
			\STATE Set $\theta^{(k+1)}$ by one Adam update using $\nabla \mathcal{L}^{M, N}_{\text{BSDE}}(\theta^{(k)})$ 
			\ENDFOR
			\RETURN $\theta^{(K)}$   
		\end{algorithmic}
	\end{algorithm}

	\section{Numerical examples} \label{sec:numerical-experiments}
	
	In this section, we test the proposed deep particle methods~\ref{alg:direct-approach} and~\ref{alg:deep-bsde} on several examples. We first consider a linear--quadratic example; the availability of analytical solutions for both the optimal control and the value function allows us to comprehensively assess numerical accuracy and compare our results against the recurrent neural network (RNN)-based method proposed in \cite{han2021recurrent}. Next, we study a partially observable MFC problem, where the drift of the observation process is nonlinear. To cross-validate the results, we solve this problem numerically using both the direct approach and the deep BSDE approach. Finally, we demonstrate the practical utility of our framework in two financial applications: portfolio liquidation and allocation. 
	For the last portfolio allocation problem, the common noise volatility term $\sigma_0$ is also controlled, so that our theoretical convergence results do not apply directly.
	Nevertheless, the numerical experiments still produce reasonable results. 
	
\subsection{A linear--quadratic problem without mean-field interactions}
\label{subsec:linear_quad}
	
	The first example aims to validate the numerical accuracy against the exact analytical solutions. We consider a linear--quadratic problem adapted from \cite{li2024discrete}. Specifically, we set the coefficient functions as 
		\begin{equation*}
			b(t,x,\alpha,\mu) = \alpha, \quad \sigma = 1, \quad h(t,x) = x, \quad f(t,x,\alpha,\mu) = \alpha^2, \quad g(x,\mu) = x^2.
		\end{equation*}
		The arguments $\mu$ are dummy variables included only to match the general notation. This benchmark lies outside the bounded-coefficient and compact-control assumptions of Theorem~\ref{thm:cvg_discrete}, but its explicit solution makes it useful for numerical validation.
		By \cite[Theorem 2.4.1]{bensoussan1992stochastic}, the optimal control admits an explicit solution given by 
	\begin{equation*}
		\alpha_t^* = -\pi_t \hat{y}_t,
	\end{equation*}
	where the filtered state $\hat{y}_t$ is governed by the SDE,  
	\begin{equation*}
		d \hat{y}_t = -(\pi_t + p_t)\hat{y}_t dt + p_t dU_t, \quad \hat{y}_0 = x_0, 
	\end{equation*}
	and $(p, \pi)$ solves the forward and backward ODEs, 
	\begin{equation*}
		\begin{aligned}
			\dot{p}_t &= 1- p_t^2, \quad p_0 = 0, \\
			\dot{\pi}_t &= \pi_t^2, \quad \pi_T = 1. 
		\end{aligned}
	\end{equation*}
	Solving this system gives the final explicit form for the optimal control 
	\begin{equation} \label{eq:lq1-explicit-optimal-control}
		\alpha_t^* = -\frac{1}{\cosh(t)}\left( \frac{x_0}{1+T} + \int_0^t \frac{\sinh(s)}{1+T-s} dU_s\right).
	\end{equation}
	
	We first evaluate the direct approach described in Section~\ref{subsec:direct-approach}. For this $1$-dimensional problem, we configure the symmetric network parameters with an input dimension $d_0=2$, hidden dimension $d_h=32$, latent dimension $d_m=10$, $d_{out}=1$, and number of hidden layers $\ell=2$. The network is initialized using the Xavier uniform method and trained using the Adam optimizer with a learning rate of $\bar{\gamma} = 10^{-3}$, a batch size of $M=128$, and $K=3000$ epochs. 
	
	\paragraph{Approximation of the optimal control.}
	
	Since the observation process $U$ is a standard Brownian motion under the reference measure $\mathbb{Q}$, we simulate three sample paths of $U$ and compute the corresponding optimal controls $\alpha_t^*$ in \eqref{eq:lq1-explicit-optimal-control}. These explicit paths are compared with the discrete control processes $(a_i(\cdot;\theta^i))_{i=0}^{N_T-1}$ generated by the trained neural networks. Figure~\ref{fig:lq1-direct-symmetric-paths-N100-T05} illustrates this comparison for an initial state $x_0 = 0$ and time horizon $T=0.5$. The top panel displays the control trajectories, while the bottom panel compares the resulting optimal state processes $X_t$. 
	
		We observe that the discrepancy between the numerical and explicit solutions for $\alpha_t^*$ is relatively noticeable, whereas the difference in the optimal state process $X_t$ remains small. This can be explained by the structure of the objective functional $\mathbb{E}[\int_0^T (\alpha_t^*)^2\,\dd t + (X_T)^2]$, in which the running cost contributes only a relatively small portion of the total cost. Consequently, during training, the neural network places less emphasis on accurately approximating the control at smaller scales, making the fine-scale approximation of $\alpha_t^*$ more challenging. We increase the number of particles $N$ to $1000$ and the number of time steps $N_T$ to $200$, and report the comparison in Figure~\ref{fig:lq1-direct-symmetric-paths-N1000-NT200-T05-value}. As shown in Figure~\ref{fig:lq1-direct-symmetric-paths-N1000-NT200-T05-value}, this denser discretization significantly improves the accuracy of the control approximation, which is empirically consistent with the expected effect of refining the particle and time discretizations. 

	For the partially observed control problem, one can also design a recurrent neural network directly to represent the path-dependent control containing the full history of observable processes,  
	and then optimize the cost function over the network parameters. This approach may avoid the need to derive a time-consistent system via particle approximation. 
	To compare it with our direct approach based on the particle system, we implement the RNN-based algorithm adapted from \cite{han2021recurrent}, originally designed for stochastic control problems with path-dependent features (both methods were allocated comparable computing time).
	The numerical results are provided in Figure~\ref{fig:lq1-direct-symmetric-paths-N100-T05} and~\ref{fig:lq1-direct-symmetric-paths-N1000-NT200-T05-value}.
		One observes that the direct approach in Algorithm \ref{alg:direct-approach} with a coarser resolution ($N=100$, $N_T=100$) produces a control error of comparable magnitude to that of the RNN-based approach, while increasing the discretization to $N=1000$ particles and $N_T=200$ time steps significantly improves its accuracy.
	Notice that, in the RNN-based method, time $t$ is included as an input to the network, which results in a learned control process that is smoother than the optimal control trajectory.

	\begin{figure}[htbp!]
		\centering
		\includegraphics[width=\linewidth]{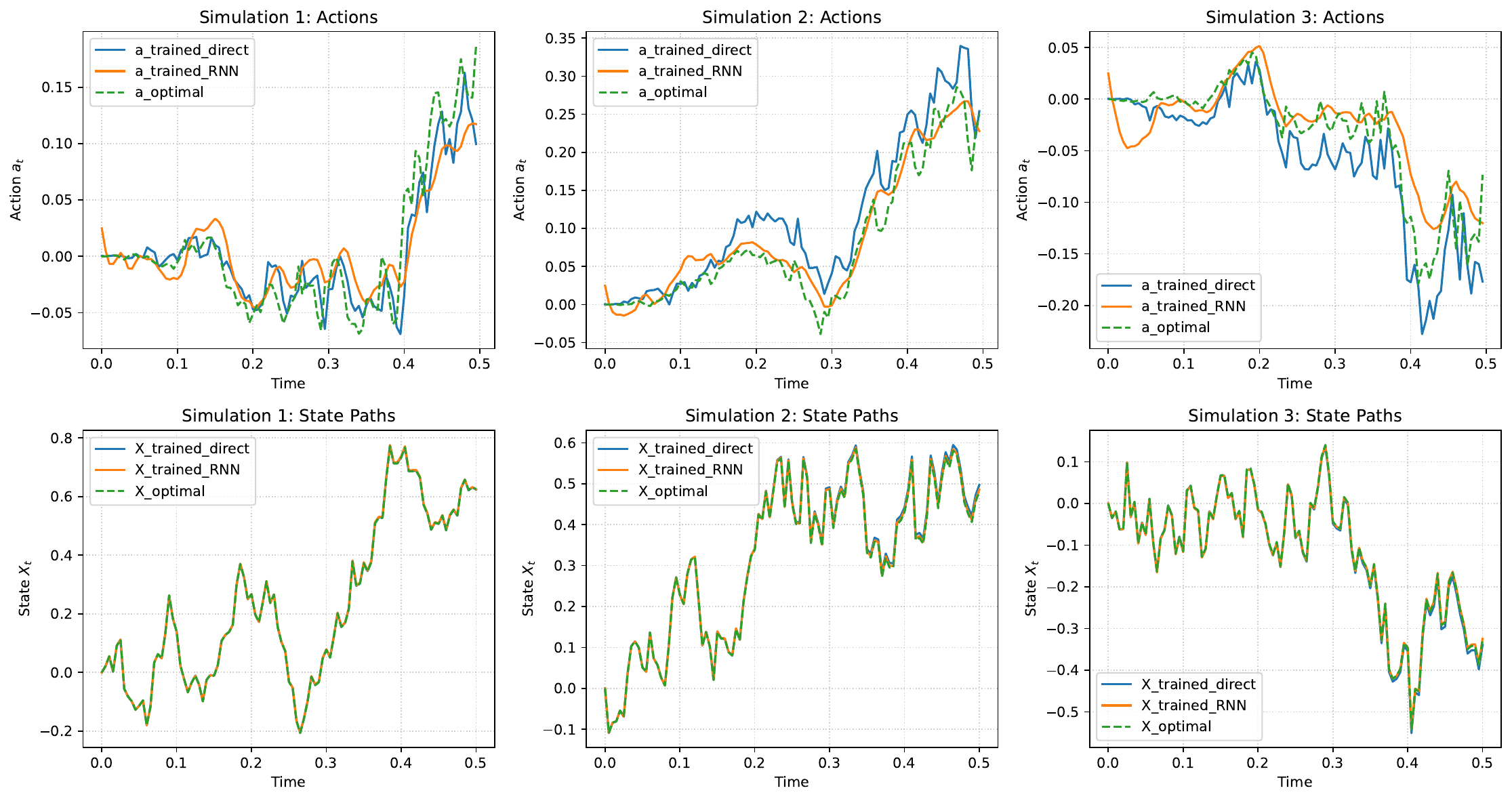}
		\caption{Comparison of direct approach solutions, RNN solutions, and explicit solutions in the simulation of optimal control $\alpha_t^*$ and the controlled optimal state $X_t$ using $N=100$ particles and $N_T=100$ time steps.}
		\label{fig:lq1-direct-symmetric-paths-N100-T05}
	\end{figure}
	
	\begin{figure}[htbp!]
		\centering
		\includegraphics[width=\linewidth]{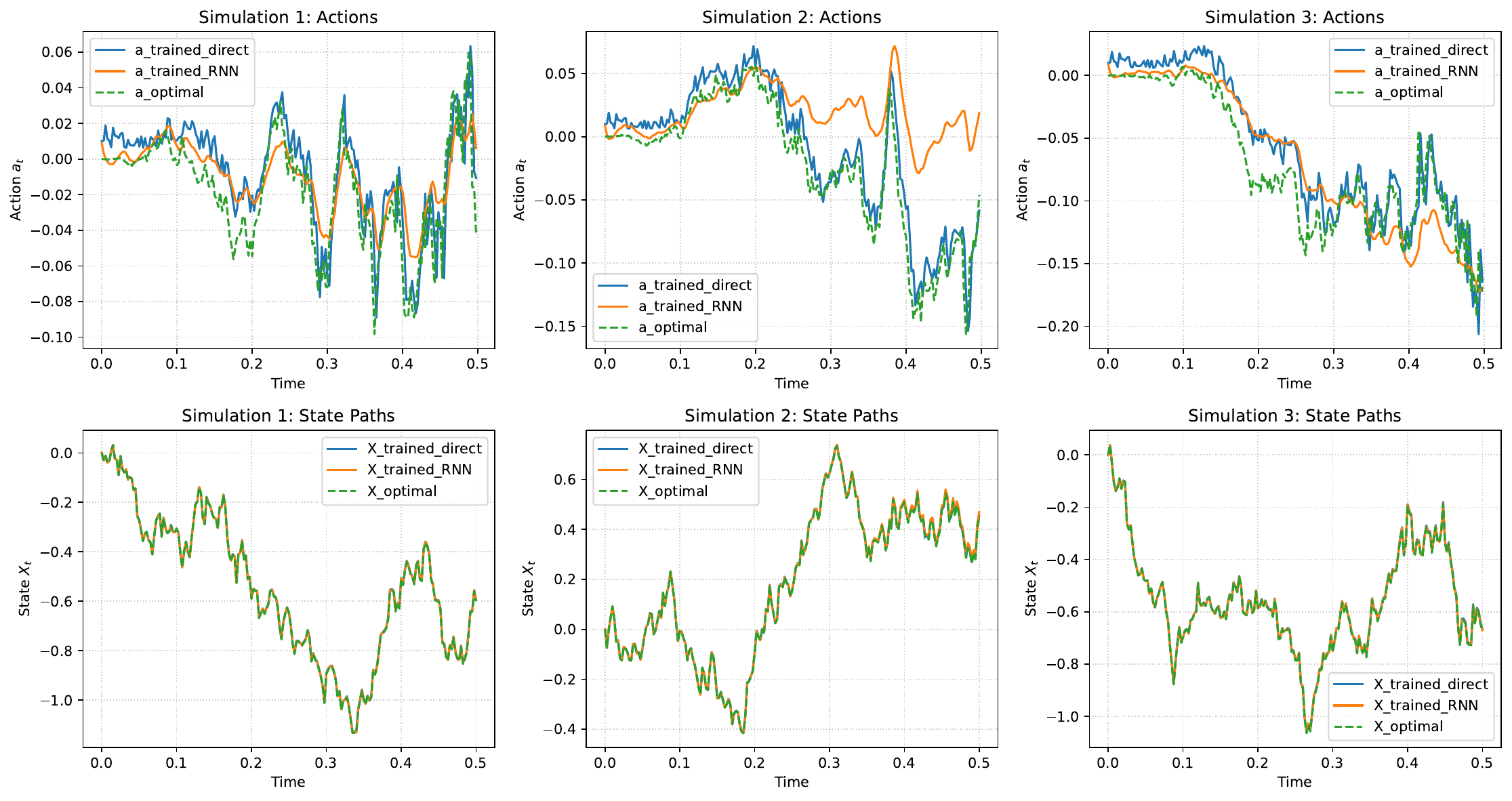}
		\caption{Comparison of direct approach solutions, RNN solutions, and explicit solutions in the simulation of optimal control $\alpha_t^*$ and the controlled optimal state $X_t$ using $N=1000$ particles and $N_T = 200$ time steps.}
		\label{fig:lq1-direct-symmetric-paths-N1000-NT200-T05-value}
	\end{figure}
	
	To quantify the approximation error of the control process, we compute the $L^2$ error on $\alpha$ by $(\frac{1}{M}\sum_{m=1}^M \sum_{i=0}^{N_T-1}|\hat{\alpha}_{t_i}^{(m)} - \alpha_{t_i}^{(m),*}|^2\Delta t)^{1/2}$, where $(\hat{\alpha}_{t_i}^{(m)})_{i=0}^{N_T-1}$ denotes the $m$-th control process computed numerically. Figure~\ref{fig:lq1_l2_error_analysis} displays the evolution of the $L^2$ error with respect to the number of training iterations, using $M=1000$ testing paths. We observe that the direct method achieves a smaller error than the RNN-based baseline under the same time discretization. 
	
	\begin{figure}[htbp!]
		\centering
		\includegraphics[width=0.5\linewidth]{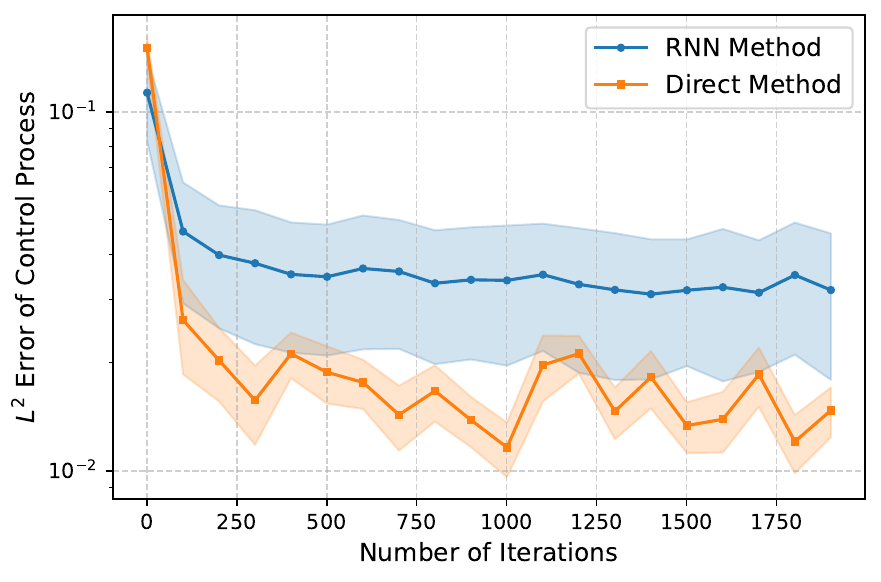}
		\caption{$L^2$ error of control process via the direct approach and the RNN-based method with $N_T = 200$ time steps.}
		\label{fig:lq1_l2_error_analysis}
	\end{figure}
	
	\paragraph{Approximation of the optimal value.}
	
	The analytical minimum total cost is given by 
	\begin{equation*}
		V_0 = \frac{x_0^2}{1+T} +\int_0^T \frac{\tanh^2(t)}{1+T-t}\,\dd t + \tanh(T). 
	\end{equation*}
	Setting $x_0=0$, $T=0.5$ and evaluating the integral term numerically, we obtain $V_0 = 0.4959$, which serves as our reference solution. 
	
	For the direct approach, the trained neural networks $(a_i(\cdot; \theta^i))_{i=0}^{N_T-1}$ provide a sequence of near-optimal feedback controls. We approximate the corresponding optimal value $V_0$ of the partially observable control problem via a Monte Carlo simulation utilizing $N_{MC}=100,000$ samples, reporting both the value estimate and standard error in Table~\ref{tab:lq1-combined-optimal-value}. 
	
		The Deep BSDE approach~\ref{alg:deep-bsde} provides an alternative algorithm. When the diffusion coefficient $\sigma$ is uncontrolled and invertible (as in the linear--quadratic case), the $N$-particle approximation is naturally associated with a FBSDE characterization. Although this problem does not satisfy the Lipschitz condition as required in Proposition~\ref{prop:bsde}, we nevertheless use the following finite-particle FBSDE as a numerical formulation:
	\begin{equation*}
		\begin{aligned}
			\dd X_t^k &= \dd W_t^k,\quad X_0^k = 0,  \\
			\dd L_t^k &= X_t^k L_t^k \,\dd U_t, \quad L_0^k = 1, \quad \text{ for }k=1,\dots,N, \\
			\dd Y_t^N &= \frac{(\sum_{k=1}^N Z_t^k)^2}{4(\frac{1}{N}\sum_{k=1}^N L_t^k)}\,\dd t + \sum_{k=1}^N Z_t^k\,\dd W_t^k + Z_t^{N+1}\,\dd U_t, \quad Y_T^N = \frac{1}{N}\sum_{k=1}^N L_T^k (X_T^k)^2. 
		\end{aligned}
	\end{equation*}
	The corresponding optimal control for the $N$-particle problem is given by 
	\begin{equation*}
		\alpha_t^{N, *} = -\frac{\sum_{k=1}^N Z_t^k}{2(\frac{1}{N}\sum_{k=1}^N L_t^k)}.
	\end{equation*}
	
		Table~\ref{tab:lq1-combined-optimal-value} summarizes the optimal value estimates for both the direct approach $\hat{V}_0$ and the deep BSDE approach across various particle sizes $N$ and time steps $N_T$. For the deep BSDE approach, we report two distinct quantities: $\bar{Y}_0^N$, which is the learned parameter approximating the initial BSDE value $Y_0^N$, and $\hat{V}_0$, which is evaluated using the near-optimal control derived from the trained networks for the $Z$ process. For fixed $N_T=200$, as $N$ increases from $10$ to $1000$, the relative error of the direct approach decreases from approximately $2.1\%$ to $0.04\%$, which is empirically consistent with convergence of the particle approximation. For fixed $N$, the direct approach shows stable errors across different $N_T$ values at larger $N$, confirming that the time discretization becomes adequate at fine resolutions. With $N=1000$ and $N_T=200$, the direct approach achieves a relative error of $0.04\%$, which is quite accurate for a deep learning-based method. The deep BSDE method shows larger errors at large $N$, which may be due to the numerical difficulty of approximating the $Z$ network with both high-dimensional input and output spaces. This suggests that the direct approach is more numerically stable for this particular problem.

	\begin{table}[htbp!]
		\centering
		\begin{tabular}{cc|cc|ccc}
			\hline
			& & \multicolumn{2}{c|}{Direct Approach} & \multicolumn{3}{c}{Deep BSDE Approach} \\
			$N$ & $N_T$ & $\hat{V}_0$ & Std. Error & $\bar{Y}_0^N$ & $\hat{V}_0$ & Std. Error \\
			\hline 
			$10$ & $50$ & 0.4858 & 0.0009 & 0.4984 & 0.4982 & 0.0010 \\
			& $100$ & 0.4875 & 0.0009 & 0.4982 & 0.4997 & 0.0010 \\
			& $200$ & 0.4857 & 0.0009 & 0.4999 & 0.5005 & 0.0010 \\
			\hline
			$100$ & $50$ & 0.4956 & 0.0004 & 0.4977 & 0.5003 & 0.0004 \\
			& $100$ & 0.4956 & 0.0004 & 0.4993 & 0.4998 & 0.0004 \\
			& $200$ & 0.4957 & 0.0004 & 0.4975 & 0.5006 & 0.0004 \\
			\hline
			$1000$ & $50$ & 0.4961 & 0.0003 & 0.4904 & 0.5074 & 0.0003 \\
			& $100$ & 0.4961 & 0.0003 & 0.4940 & 0.5098 & 0.0003 \\
			& $200$ & 0.4961 & 0.0003 & 0.4843 & 0.5114 & 0.0003  \\
			\hline
		\end{tabular}
		\caption{Comparison of optimal value estimates using the direct approach and the deep BSDE approach. The exact analytical reference value is $V_0 = 0.4959$.}
		\label{tab:lq1-combined-optimal-value}
	\end{table}

	\subsection{A partially observable mean-field control problem}
	To show that the proposed particle methods can be naturally extended to the partially observable MFC problem, we consider the following problem with partial observation 
	\begin{equation*}
		\inf_{\alpha \in \mathcal{A}^U} \mathbb{E}^{\mathbb{P}}\left[ \int_0^T \left(\Var(X_t|\mathcal{F}_t^U) + |\alpha_t|^2 \right)\dd t + \Var(X_T|\mathcal{F}_T^U) \right], 
	\end{equation*}
	where the state dynamics and the observation process are governed by 
	\begin{equation*}
		\begin{aligned}
			\dd X_t &= (X_t - \mathbb{E}^\mathbb{P}[X_t|\mathcal{F}_t^U] + \alpha_t)\,\dd t + \sigma \,\dd W_t, \quad X_0 = 0,  \\
			\dd U_t &= \sin(|X_t|^2)\,\dd t + \dd W_t^0, \quad U_0 = 0, 
		\end{aligned}
	\end{equation*}
	and we set the parameters to $T=0.3$, $\sigma = 0.4$. 
	Motivated by the setup in \cite{wan2025discrete}, we incorporate a nonlinear drift $\sin(|X_t|^2)$ in the dynamic of $U$, which removes the problem from the standard linear--quadratic case. Because this configuration lacks a known analytical solution, we shall make a comparison between different numerical schemes to validate the numerical results. By the tower property of conditional expectation, we define the running and terminal cost functions as 
	\begin{equation*}
		f(t,X_t, \mu_t, \alpha_t) = |X_t|^2 - (\mathbb{E}^\mathbb{P}[X_t|\mathcal{F}_t^U])^2 + |\alpha_t|^2, \quad g(X_T, \mu_T) = |X_T|^2 - (\mathbb{E}^{\mathbb{P}}[X_T|\mathcal{F}_T^U])^2. 
	\end{equation*}
	
		We approximate the partially observable problem using an $N$-particle system. Let $\bar{X}_t^N = \sum_{k=1}^N w_t^k X_t^k$ denote the empirical weighted mean with weights $w_t^k = \frac{L_t^k}{\sum_{j=1}^N L_t^j}$. In this mean-field control example, we use the finite-particle Hamiltonian/BSDE formulation as part of the numerical scheme, rather than as a direct application of Proposition~\ref{prop:bsde}. Since $\sigma_0 = 0$, the vector $\boldsymbol{\theta}$ reduces to
	\begin{equation*}
		\boldsymbol{\theta}(t,\mathbf{X}_t,\mathbf{L}_t,\alpha_t)
		= \begin{pmatrix}
			\sigma^{-1}(X_t^1 - \bar{X}_t^N + \alpha_t) \\
			\vdots \\
			\sigma^{-1}(X_t^N - \bar{X}_t^N + \alpha_t) \\
			0
		\end{pmatrix}.
	\end{equation*}
	The pre-minimised Hamiltonian is given by 
	\begin{equation*}
		\begin{aligned}
			H(t,\mathbf{X}_t,\mathbf{L}_t,\mathbf{Z}_t,\alpha_t)
			&= \frac{1}{N}\sum_{k=1}^N L_t^k
			\Bigl(|X_t^k|^2 - (\bar{X}_t^N)^2 + |\alpha_t|^2\Bigr)
			+ \frac{1}{\sigma}\sum_{k=1}^N Z_t^k
			\Bigl(X_t^k - \bar{X}_t^N + \alpha_t\Bigr).
		\end{aligned}
	\end{equation*}
		Here the control set is taken to be $A=\mathbb{R}$, so this example is outside the compact-control assumptions of Theorem~\ref{thm:cvg_discrete}. Since $H$ is strictly convex and quadratic in $\alpha_t \in \mathbb{R}$,
	the minimizer is obtained by setting $\partial_{\alpha} H = 0$, which yields  
	\begin{equation*}
		\frac{2}{N}\Bigl(\sum_{k=1}^N L_t^k\Bigr) \alpha_t^*
		+ \frac{1}{\sigma}\sum_{k=1}^N Z_t^k = 0
		\quad\implies\quad
		\alpha_t^* = -\frac{N}{2\sigma \sum_{k=1}^N L_t^k}\sum_{k=1}^N Z_t^k.
	\end{equation*}
	Substituting $\alpha_t^*$ back gives the minimised Hamiltonian 
	\begin{equation*}
		\begin{aligned}
			\mathcal{H}(t,\mathbf{X}_t,\mathbf{L}_t,\mathbf{Z}_t)
			&= \frac{1}{N}\sum_{k=1}^N L_t^k
			\Bigl(|X_t^k|^2 - (\bar{X}_t^N)^2\Bigr)
			+ \frac{1}{\sigma}\sum_{k=1}^N Z_t^k
			\Bigl(X_t^k - \bar{X}_t^N\Bigr)
			- \frac{N}{4\sigma^2 \sum_{k=1}^N L_t^k}
			\Bigl(\sum_{k=1}^N Z_t^k\Bigr)^2.
		\end{aligned}
	\end{equation*}
	Thus, the uncontrolled $N$-particle forward system reads, for $k = 1, \dots, N$, 
	\begin{equation} \label{eq:ex2-fbsde-forward}
		\begin{aligned}
			\dd X_t^k &= \sigma \,\dd W_t^k, \quad X_0^k = 0, \\
			\dd L_t^k &= \sin(|X_t^k|^2)L_t^k\,\dd U_t, \quad L_0^k = 1, 
		\end{aligned}
	\end{equation}
	where $W^1, \dots, W^N$ and $U$ are independent Brownian motions. The value process $Y^N$ satisfies the BSDE 
	\begin{equation} \label{eq:ex2-fbsde-backward}
		\dd Y_t^N = -\mathcal{H}(t, \mathbf{X}_t,\mathbf{L}_t,\mathbf{Z}_t)\,\dd t + \sum_{k=1}^N Z_t^k \,\dd W_t^k + Z_t^{N+1}\,\dd U_t,
	\end{equation}
	with the terminal condition 
	\begin{equation*}
		Y_T^N = \frac{1}{N}\sum_{k=1}^N L_T^k
		\left(|X_T^k|^2 - (\bar{X}_T^N)^2\right).
	\end{equation*}
	
		For numerical stability, we use the following modified FBSDE system:
	\begin{equation} \label{eq:ex2-fbsde-numerics}
		\begin{aligned}
			\dd X_t^k &= (X_t^k - \bar{X}_t^N)\,\dd t + \sigma \,\dd W_t^k, \quad X_0^k = 0, \\
			\dd L_t^k &= \sin(|X_t^k|^2)L_t^k\,\dd U_t, \quad L_0^k=1, \\
			\dd Y_t^N &= -\overline{f}(t, \mathbf{X}_t, \mathbf{L}_t, \mathbf{Z}_t)\,\dd t + \sum_{k=1}^N Z_t^k \,\dd W_t^k + Z_t^{N+1}\,\dd U_t,
		\end{aligned}
	\end{equation}
	with the terminal condition $Y_T^N = \frac{1}{N}\sum_{k=1}^N L_T^k \left(|X_T^k|^2 - (\bar{X}_T^N)^2\right)$, where the modified generator is 
	\begin{equation*}
		\overline{f}(t, \mathbf{X}_t, \mathbf{L}_t, \mathbf{Z}_t) = \frac{1}{N}\sum_{k=1}^N L_t^k
		\Bigl(|X_t^k|^2 - (\bar{X}_t^N)^2\Bigr)
		- \frac{N}{4\sigma^2 \sum_{k=1}^N L_t^k}
		\Bigl(\sum_{k=1}^N Z_t^k\Bigr)^2. 
	\end{equation*}
	
		We solve the FBSDE \eqref{eq:ex2-fbsde-numerics} instead of \eqref{eq:ex2-fbsde-forward}-\eqref{eq:ex2-fbsde-backward} to improve numerical stability. This modification is crucial for mitigating distributional shift, which is common when solving stochastic control problems via deep BSDE methods. If the neural networks are trained using the purely noise-driven forward system \eqref{eq:ex2-fbsde-forward}, the simulated state trajectories remain tightly concentrated near the origin. However, under the physical dynamics, due to the individual intrinsic drift $X_t^k - \bar{X}_t^N$, the state trajectories expand and occupy a significantly wider spatial domain over time. Evaluating a network trained strictly on the narrow distribution of \eqref{eq:ex2-fbsde-forward} against these expanded physical paths will force the model to extrapolate out-of-distribution. Therefore, for the numerical implementation, we incorporate the intrinsic uncontrolled part $X_t^k - \bar{X}_t^N$ into the forward dynamics. The alignment of training and evaluation distributions improves network convergence and leads to a tight agreement between the learned initial value and the subsequent Monte Carlo estimate using near-optimal controls, see Table~\ref{tab:lq2-comparison} below. Here, we configure the symmetric network parameters with an input dimension $d_0=2$, hidden dimension $d_h=32$, latent dimension $d_m=10$, $d_{out}=1$, and number of hidden layers $\ell=2$. The network is initialized using the Xavier uniform method and trained using the Adam optimizer with a learning rate of $\bar{\gamma} = 10^{-3}$, a batch size of $M=64$, and $K=1500$ epochs. 
	
	To verify the numerical accuracy in the absence of an explicit analytical solution, we compare the proposed direct method against the deep BSDE approach. Both methods utilize symmetric neural networks, but they differ fundamentally in their approximation quantities and loss formulations. The direct approach approximates the control processes and minimizes cost functional directly, while the deep BSDE approach treats the initial value $V_0^N$ and the $Z$ process in the BSDE as learnable parameters to match the target terminal condition. The close agreement between these two independent algorithms serves as strong evidence of numerical accuracy. 
	
	For the deep BSDE approach, we use symmetric networks with $d_{\mathrm{out}}=N$ to represent $(Z_{t_i}^1,\dots,Z_{t_i}^N)$, together with additional symmetric networks with $d_{\mathrm{out}}=1$ for $Z_{t_i}^{N+1}$. All other hyperparameters are chosen to be the same as those in the direct approach to ensure a fair comparison. 
	
	Table~\ref{tab:lq2-comparison} summarizes the numerical results across various particle sizes $N$ and time steps $N_T$. For the deep BSDE approach, we report two distinct quantities: $\bar{Y}_0^N$, which is the learned parameter approximating the initial BSDE value $Y_0^N$, and $\hat{V}_0$, which is evaluated using the near-optimal control derived from the trained networks for the $Z$ process. The standard errors for both approaches are computed via Monte Carlo simulations, utilizing these learned controls over 100,000 sample paths. In the absence of an exact analytical benchmark, the close agreement between the direct approach and the deep BSDE method serves to mutually validate the numerical integrity of our solutions. 
	
	\begin{table}[htbp!]
		\centering
		\begin{tabular}{cc|cc|ccc}
			\hline
			& & \multicolumn{2}{c|}{Direct Approach} & \multicolumn{3}{c}{Deep BSDE Approach} \\
			$N$ & $N_T$ & $\hat{V}_0$ & Std. Error & $\bar{Y}_0^N$ & $\hat{V}_0$ & Std. Error \\
			\hline
			$10$ & $50$ & 0.0664 & 0.00010  & 0.0676 & 0.0665 & 0.00010 \\
			& $100$ & 0.0668 & 0.00010  & 0.0676 & 0.0670 & 0.00010 \\
			& $200$ & 0.0671 & 0.00010  & 0.0665 & 0.0672 & 0.00010 \\
			\hline
			$100$ & $50$ & 0.0732 & 0.00003  & 0.0735 & 0.0732 & 0.00003 \\
			& $100$ & 0.0736 & 0.00003  & 0.0731 & 0.0736 & 0.00003 \\
			& $200$ & 0.0738 & 0.00003  & 0.0742 & 0.0737 & 0.00003 \\
			\hline
			$1000$ & $50$ & 0.0739 & 0.00001 & 0.0739 & 0.0739 & 0.00001 \\
			& $100$ & 0.0742 & 0.00001 & 0.0742 & 0.0742 & 0.00001 \\
			& $200$ & 0.0744 & 0.00001 & 0.0743 & 0.0744 & 0.00001 \\
			\hline
		\end{tabular}
		\caption{Comparison of optimal value estimates between the direct approach and the deep BSDE approach.}
		\label{tab:lq2-comparison}
	\end{table}
	
	Table~\ref{tab:lq2-comparison} demonstrates several observations. First, for a fixed particle size $N$, the value estimates stabilize as we refine the time discretization, indicating convergence with respect to the time step. Second, for large $N = 1000$, the two methods (direct and deep BSDE) produce nearly identical results, with relative differences less than $0.7\%$, providing strong mutual validation of accuracy. Third, the estimates improve and stabilize as we increase from $N=10$ to $N=1000$. This agreement between two fundamentally different algorithms verifies that our numerical solutions are reliable, despite the absence of an analytical benchmark for this nonlinear partially observable mean-field control problem.

	\subsection{Portfolio optimization under drift uncertainty} \label{subsec:portfolio}
		We consider portfolio optimization problems under drift uncertainty adapted from \cite{balata2019class}, which can be viewed as partially observable stochastic control problems. In standard portfolio optimization, an investor typically assumes full knowledge of a stock's expected return (drift) and volatility. In practice, while volatility can be reliably estimated from high-frequency historical data, the expected return remains difficult to observe since it changes dynamically in response to business cycles and other macroeconomic factors. In the following we consider two types of optimization problems in finance. Both examples demonstrate the practical applicability of our numerical methods, but they also go beyond the literal scope of Theorem~\ref{thm:cvg_discrete}: the liquidation model involves unbounded state and control features, while the allocation problem has controlled common-noise volatility. They should therefore be viewed as numerical demonstrations rather than theoretically guaranteed applications of the convergence theorem.
	
	\paragraph{Portfolio liquidation problem.}
	
	We consider the problem of an investor who decides to liquidate the portfolio, i.e., sell the asset (considering one stock) and end up holding mostly cash by time $T$. The investor controls the trading rate (or velocity), denoted by $\alpha_t$, at which he/she buys ($\alpha_t>0$) or sells $(\alpha_t<0)$ the stock. Consequently, the inventory $q_t$, representing the number of shares held at time $t$, evolves according to 
	\begin{equation*}
		\dd q_t = \alpha_t \,\dd t, 
	\end{equation*}
	subject to the initial inventory $q_0$. The investor can make decisions by observing the stock price governed by 
	\begin{equation*}
		\dd S_t = \beta_t S_t \,\dd t + \sigma_S S_t\,\dd B_t, \quad S_0 = s_0.
	\end{equation*}
	Here, the drift process $\beta$ is partially observable, which can not be observed directly but is assumed to be a mean-reverting stochastic process driven by 
	\begin{equation*}
		\dd \beta_t = \kappa (\bar{\beta} - \beta_t)\,\dd t + \sigma_\beta \dd W_t, \quad \beta_0 = b_0,  
	\end{equation*}
	where $W$ and $B$ are independent Brownian motions. The trader's objective is to minimize the expected total liquidation cost, formulated as 
	\begin{equation*}
		J(\alpha) = \mathbb{E}^{\mathbb{P}}\left[ \int_0^T \alpha_t(S_t + \gamma \alpha_t) \,\dd t + \eta q_T^2 \right], 
	\end{equation*}
	where $\gamma \alpha_t$ captures the price impact at time $t$ and $\eta q_T^2$ performs as a penalization term ensuring the liquidation is completed by time $T$. 
	
	To fit in our framework, we introduce the log-price observation process $U_t = \frac{\ln S_t}{\sigma_S}$ for $t>0$. By It\^{o}'s lemma, $U$ follows 
	\begin{equation*}
		\dd U_t = h(t, \beta_t)\,\dd t + \dd B_t, 
	\end{equation*}
	where $h(t, \beta_t) = \frac{\beta_t}{\sigma_S} - \frac{\sigma_S}{2}$. 
	Let $X_t = (\beta_t, q_t, U_t)^\top$ be the hidden state process given by 
	\begin{equation*}
		\dd\begin{pmatrix}
			\beta_t \\
			q_t \\
			U_t
		\end{pmatrix} = \begin{pmatrix}
			\kappa(\bar{\beta} - \beta_t) \\
			\alpha_t \\
			h(t,\beta_t)
		\end{pmatrix} \,\dd t + \begin{pmatrix}
			\sigma_\beta \\
			0 \\
			0
		\end{pmatrix}\,\dd W_t + \begin{pmatrix}
			0 \\
			0 \\
			1
		\end{pmatrix} \,\dd B_t
	\end{equation*}
	Under this transformation, the trader's optimization problem can be equivalently rewritten as 
	\begin{equation} \label{eq:ex3-liquidation-objective}
		\inf_{\alpha \in \mathcal{A}^U} \mathbb{E}^{\mathbb{P}}\left[ \int_0^T \left(\alpha_t \exp(\sigma_S U_t) + \gamma \alpha_t^2 \right) \,\dd t + \eta q_T^2 \right], 
	\end{equation}
	where $\mathcal{A}^U$ denotes the set of admissible trading strategies adapted to the observation filtration generated by $U$. 
	
	We solve the partial observation control problem \eqref{eq:ex3-liquidation-objective} using the direct approach for its $N$-particle approximation. For numerical experiments, we configure the symmetric network parameters with an input dimension $d_0=4$ (for $\beta$, $q$, $U$, and $L$), hidden dimension $d_h=32$, latent dimension $d_m=10$, $d_{out}=1$, and number of hidden layers $\ell=2$. The network is initialized using the Xavier uniform method and trained using the Adam optimizer with a learning rate of $\bar{\gamma} = 3\times 10^{-5}$, a batch size of $M=64$, and $K=3000$ epochs. Set $T=1.5$, $\kappa = 0.03$, $\bar{\beta}=0.1$, $\sigma_S = 0.4$, $s_0=6$, $\beta_0=0.03$, $q_0 = 1$, $\gamma = 5$, and $\eta=100$. We discretize the problem with $N=1000$ particles and $N_T=100$ time steps. 
	
	To investigate the sensitivity of the optimal strategy to underlying market dynamics, we train and evaluate the model across a range of values for $\sigma_\beta$, the volatility of the unobserved drift. As illustrated in Figure~\ref{fig:ex3_1_comparison_po_fo_expected_cost_vs_sigma_beta}, the optimal expected execution cost decreases monotonically as $\sigma_\beta$ increases. Intuitively, a higher volatility in the unobserved drift implies larger transient trends in the asset price. The trained strategy can take advantage of these fluctuations and opportunistically change the trading rate, thereby exploiting favorable price movements to minimize the overall execution cost. 
	
	In this problem, we assume the drift $\beta$ follows an unobserved mean-reverting process. To demonstrate the impact of the partial observation feature, Figure~\ref{fig:ex3_1_comparison_po_fo_expected_cost_vs_sigma_beta} also compares the optimal value of the partially observable problem with a fully observable counterpart, where the investor can observe the drift process $\beta$ for decision-making. We can see that with more information to the investor, the expected optimal cost is always smaller. 
	
	\begin{figure}[htbp!]
		\centering
		\includegraphics[width=0.5\linewidth]{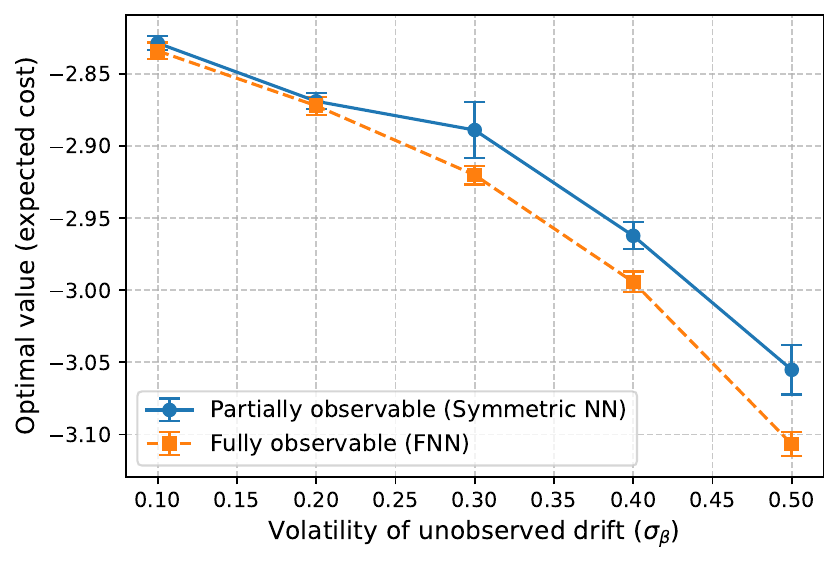}
		\caption{The optimal expected cost against the volatility of unobserved drift for the liquidation problem.}
		\label{fig:ex3_1_comparison_po_fo_expected_cost_vs_sigma_beta}
	\end{figure}
	
	Figure~\ref{fig:ex3_1_qt_comparison} visualizes the inventory trajectories $q_t$ under three different volatility settings: $\sigma_\beta = 0.1$, $0.5$, and $1.0$. The inventory paths closely track a trivial strategy, which consists of liquidating the inventory at a constant, uniform rate $\alpha_t = -q_0/T$. Across various $\sigma_\beta$ configurations, the trained optimal strategy consistently achieves a lower expected liquidation cost than the trivial strategy. 
	
	\begin{figure}[htbp!]
		\centering
		\includegraphics[width=0.5\linewidth]{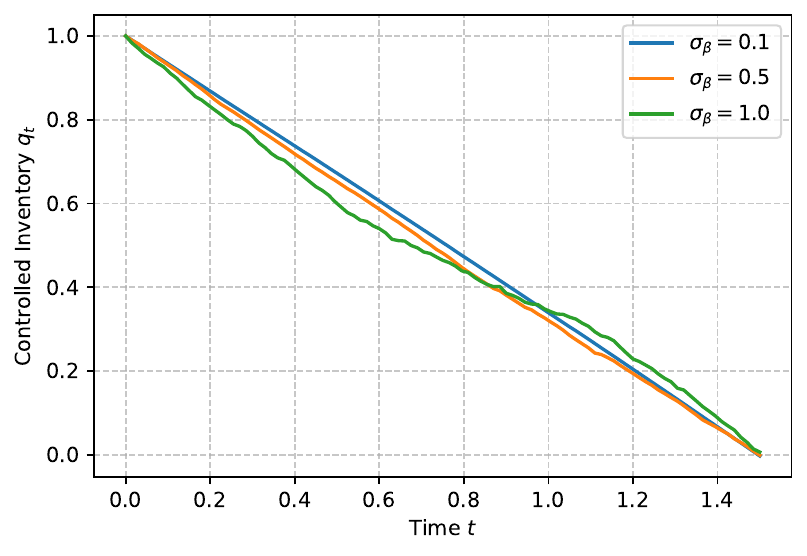}
		\caption{The comparison of the inventory $q_t$ between $\sigma_\beta=0.1, 0.5, 1.0$.}
		\label{fig:ex3_1_qt_comparison}
	\end{figure}
	
	\paragraph{Portfolio allocation problem.} 
	
	Consider a portfolio consisting of a stock with the price $S_t$ and a risk-free asset yielding a constant interest rate $r$. An investor with initial wealth $w_0$ dynamically allocates $\alpha_t$ dollars to the risky asset ($\alpha_t<0$ corresponds to taking a short position in the stock), keeping the residual wealth in the risk-free account. Consequently, the self-financing wealth process $w_t$ follows the dynamic 
	\begin{equation*}
		\begin{aligned}
			\dd w_t &= r(w_t-\alpha_t)\,\dd t + \alpha_t(\beta_t \,\dd t + \sigma_S \,\dd B_t) \\
			&= \left( rw_t + \alpha_t (\beta_t - r) \right)\,\dd t + \sigma_S \alpha_t \,\dd B_t. 
		\end{aligned}
	\end{equation*}
	The investor's objective is to maximize the expected utility of the terminal wealth $w_T$:  
	\begin{equation*}
		\max_{\alpha} \mathbb{E}^{\mathbb{P}}[u(w_T)], 
	\end{equation*}
	given only the information generated by observing the stock price. Here we use a CARA utility function $u$ as $u(w) = -\exp(-pw), \quad p>0$.
	
	To cast this into our partial observation framework, we define the observation process as $U_t = \frac{\ln S_t}{\sigma_S}$, which satisfies 
	\begin{equation*}
		\dd U_t = h(t,\beta_t)\,\dd t + \dd B_t. 
	\end{equation*}
	Let $X_t = (\beta_t, w_t)$ be the hidden state process given by 
	\begin{equation*}
		\dd\begin{pmatrix}
			\beta_t \\
			w_t
		\end{pmatrix} = \begin{pmatrix}
			\kappa(\bar{\beta} - \beta_t) \\
			rw_t + \alpha_t (\beta_t - r)
		\end{pmatrix} \,\dd t + \begin{pmatrix}
			\sigma_\beta \\
			0
		\end{pmatrix}\,\dd W_t + \begin{pmatrix}
			0 \\
			\sigma_S \alpha_t
		\end{pmatrix} \,\dd B_t. 
	\end{equation*}
	
		Although this problem falls outside our theoretical framework since the volatility $\sigma_0$ is controlled and the allocation and wealth variables are unbounded, we nevertheless apply our proposed particle method 
		with the direct deep learning approach described in Algorithm \ref{alg:direct-approach}. We configure the symmetric network parameters with an input dimension $d_0=3$ (for $\beta, w$, and $L$), hidden dimension $d_h=32$, latent dimension $d_m=10$, $d_{out}=1$, and number of hidden layers $\ell=2$. The network is initialized using the Xavier uniform method and trained using the Adam optimizer with a learning rate of $\bar{\gamma} = 10^{-3}$, a batch size of $M=128$, and $K=1000$ epochs. For the problem coefficients, we set $T=1$, $p=1$, $\kappa = 0.5$, $\bar{\beta}=0.1$, $\sigma_\beta = 0.6$, $\sigma_S = 0.4$, $s_0=6$, $\beta_0=0.08$, $r=0.03$, and $w_0 = 0$. When $w_0=0$, the investor can borrow cash at the risk-free rate $r$ to finance the purchase of the stock. Figure~\ref{fig:ex3_2_direct_symmetric_nn_paths} illustrates $10$ simulated paths of the trained optimal control process and the corresponding wealth process under the physical measure $\mathbb{P}$. These paths provide a qualitative illustration of the learned strategy and wealth dynamics, rather than a separate verification of optimality. 
	
	\begin{figure}[htbp!]
		\centering
		\includegraphics[width=\linewidth]{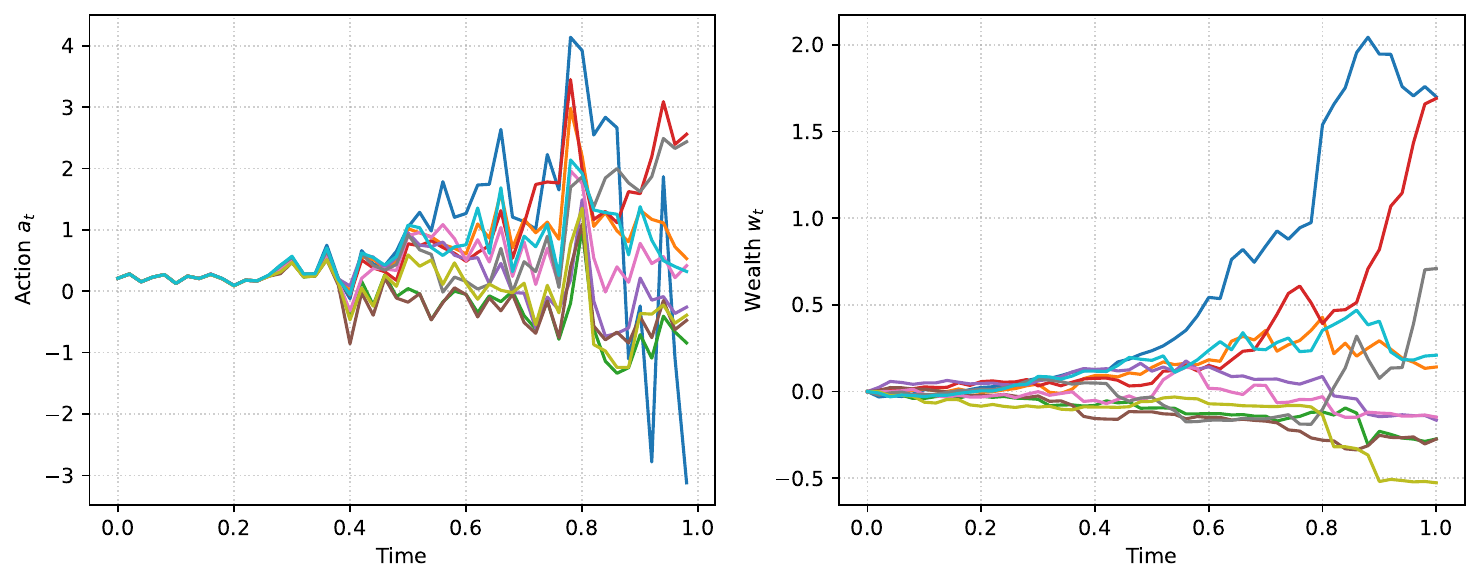}
		\caption{Simulations of the trained optimal control $\alpha_t$ and the controlled wealth process $w_t$.}
		\label{fig:ex3_2_direct_symmetric_nn_paths}
	\end{figure}
	
	\section{Conclusion}\label{sec:ccl}
		This paper addresses the challenge of numerically solving stochastic control problems under partial observation by proposing particle approximations combined with deep learning. Under the assumptions of Theorem~\ref{thm:cvg_discrete}, we establish convergence of the fully discretized $N$-particle system with time step $\Delta t$ to the original continuous-time partially observable control problem as $N\to \infty$ and $\Delta t \to 0$. The particle approximation naturally extends to partially observable mean-field control problems, a class that has been studied theoretically but remains seldom addressed numerically. Comprehensive experiments on a linear--quadratic benchmark with analytical solutions, a nonlinear mean field control problem, and financial applications demonstrate the practical utility and accuracy of the proposed approach. 
	
		We now mention some limitations and possible future directions. The current convergence theory assumes compact controls, bounded coefficients, suitable continuity of the costs, and uncontrolled volatility coefficient $\sigma_0$ and observation drift $h$. These assumptions exclude some of the numerical examples in their present form, especially the financial applications in Section~\ref{subsec:portfolio}. While our numerical experiments show that the methods work heuristically in such cases, extending the theoretical framework to cover these settings is an important direction. Additionally, although we present a heuristic extension to mean-field control problems, the rigorous convergence analysis for mean-field control problems remains open and is left for future investigation. 
		
	\section*{Acknowledgements}
	This work was supported in part through the NYU Shanghai IT High Performance Computing resources, services, and staff expertise. 
	
	\bibliographystyle{siam} 
	\bibliography{references}
\end{document}